\documentclass[12pt]{amsart}
\usepackage{amssymb,latexsym,amsmath,epsfig}
\usepackage[hyphens]{url} 
\usepackage[colorlinks,linkcolor=blue]{hyperref}
\usepackage[utf8]{inputenc}
\usepackage[T1]{fontenc}
\usepackage{graphicx}
\usepackage{mathrsfs}

 \hypersetup{colorlinks = true,linkcolor = blue,anchorcolor =red,citecolor = blue,filecolor = red,urlcolor = red,pdfauthor=author}

\def\pf{\begin{proof}}

	\usepackage{amssymb,textcomp, url}
	
\usepackage{geometry}

 \newtheorem{thm}{Theorem}[section]
 \newtheorem{prop}[thm]{Proposition}
 \newtheorem{lem}[thm]{Lemma}
 \newtheorem{cor}[thm]{Corollary}
\theoremstyle{definition}
 
 \newtheorem{defn}[thm]{Definition}

 \newtheorem{rem}[thm]{Remark}
 \numberwithin{equation}{section}

 \usepackage{t1enc,mathrsfs,amssymb,amsfonts,latexsym,pifont,fancybox}
	\newcommand{\bprop} {\begin{proposition}}
		\newcommand{\eprop} {\end{proposition}}
	\newcommand{\btheo} {\begin{theorem}}
		\newcommand{\etheo} {\end{theorem}}
	\newcommand{\blem} {\begin{lemma}}
		\newcommand{\elem} {\end{lemma}}
	\newcommand{\bcor} {\begin{corollary}}
		\newcommand{\ecor} {\end{corollary}}
	
	\newcounter{rea}
	\newcounter{red}
	\newcommand{\Be}{\begin{equation}}
	\newcommand{\Ee}{\end{equation}}
	\newcommand{\Bea}{\begin{eqnarray}}
	\newcommand{\Eea}{\end{eqnarray}}
	\newcommand{\Bes}{\begin{equation*}}
	\newcommand{\Ees}{\end{equation*}}
	\newcommand{\Beas}{\begin{eqnarray*}}
		\newcommand{\Eeas}{\end{eqnarray*}}
	\newcommand{\Ba}{\begin{array}}
		\newcommand{\Ea}{\end{array}}

	\scrollmode

	\title[ Toeplitz Operators  on $\ell^2-$Valued Bergman Spaces.] {Analysis of Toeplitz Operators with $BMO^1_{\alpha}$ operator-valued symbols on $\ell^2-$Valued Bergman Spaces.  
	}



\begin{document}

		\begin{abstract}
As a class of compact operators on the $\ell^2-$valued Bergman space $A^2_\alpha (\mathbb B_n, \ell^2)$ on the unit ball $\mathbb B_n,$ we study Toeplitz operators with $BMO^1_\alpha (\mathbb B_n, \mathcal L(\ell^2))$ operator-valued symbols. First, we describe a method of restriction to a finite dimension which allows us to apply earlier results of Rahm and Wick; then we exhibit an explicit example of a compact Toeplitz operator on $A^2_\alpha (\mathbb B_n, \ell^2).$ Secondly, we apply two sufficient conditions for compactness established by Rahm in infinite dimension. The first condition is in terms of the Toeplitz algebra $\mathcal T_{L^\infty_{fin}},$ the second one is in terms of sufficiently localized operators and is implied by the first condition. To get the second condition, we additionally assume that the symbol and its adjoint belong to $BMO^1_\alpha (\mathbb B_n, \mathcal L(\ell^2, \ell^1)).$ Finally, inspired by Xia and Sadeghi-Zorboska, we ask the question of the validity of the reverse implication.
		\end{abstract}

			\author[D. B\'ekoll\`e]{David B\'ekoll\`e}
			\address{Department of Mathematics, Faculty of Science, University of Yaound\'e I, P.O. Box 812, Yaound\'e, Cameroon }
			\email{{\tt david.bekolle@univ-yaounde1.cm \& dbekolle@gmail.com}}
			\author[H. O. Defo]{Hugues Olivier Defo}
			\address{Department of Mathematics, Faculty of Science, University of Yaound\'e I, P.O. Box 812, Yaound\'e, Cameroon }
			\email{{\tt hugues.defo@facsciences-uy1.cm \& deffohugues@gmail.com}}
			\author[E. L. Tchoundja]{Edgar L. Tchoundja}
			\address{Department of Mathematics, Faculty of Science, University of Yaound\'e I, P.O. Box 812, Yaound\'e, Cameroon }
			\email{\tt edgartchoundja@gmail.com}
		
		\maketitle

		\section{Introduction}\label{sec1}
Throughout this paper, we fix a nonnegative integer $n$ and let $\mathbb{C}^n = \mathbb{C} \times \cdots \times \mathbb{C}$ denote the $n-$dimensional complex Euclidean space. For $z = (z_{1},\cdots,z_{n})$ and $w = (w_{1},\cdots,w_{n})$ in $\mathbb{C}^{n},$ we define the inner product of $z$ and $w$ by $ \langle z,w \rangle = z_{1}\overline{w_{1}} + \cdots + z_{n}\overline{w_{n}},$ where $\overline{w_{k}}$ is the complex conjugate of $w_{k}.$ The resulting norm is then $ |z| = \sqrt{\langle z,z \rangle} = \sqrt{|z_{1}|^2 + \cdots + |z_{n}|^2 }.$ 
The open unit ball in $\mathbb{C}^n$ is the set $\mathbb{B}_n = \lbrace z \in \mathbb{C}^{n}: |z| < 1 \rbrace.$
When $\alpha >-1,$ the weighted Lebesgue measure $\mathrm{d}\nu_{\alpha}$ in $\mathbb{B}_n$ is defined by $\mathrm{d}\nu_{\alpha}(z) = c_{\alpha}(1-|z|^2)^{\alpha}\mathrm{d}\nu(z), $ where $z \in \mathbb{B}_{n},$ $\mathrm{d}\nu$ is the Lebesgue measure in $\mathbb{C}^n$ and  $c_{\alpha} = \frac{\Gamma(n+\alpha+1)}{n!\Gamma(\alpha+1)}$ is the normalizing constant so that $\mathrm{d}\nu_{\alpha}$ becomes a probability measure on $\mathbb{B}_n.$  Throughout this paper, we fix $\alpha \textgreater -1.$ \\
Let $E$ be a complex Banach space. A $E-$valued function on the unit ball ${\mathbb B}_n$ of $\mathbb C^n$ is a function $f: \mathbb B_n \rightarrow E.$ The integral of a vector-valued function that we are going to use has been introduced by Bochner in \cite{B1933}. The Bochner integral is defined similarly as the Lebesgue integral (see \cite{D2001} for more information). We recall the following theorem.
\begin{thm}[Hille's Theorem]\label{thm11}
Let $E$ and $F$ be two complex Banach spaces and let $T: E\rightarrow F$ be a continuous linear operator. Let $\mu$ be a positive $\sigma-$finite Borel measure on ${\mathbb B}_n.$ If $f: {\mathbb B}_n \rightarrow E$ is $\mu-$Bochner integrable, then $T(f)$ is also $\mu-$Bochner integrable and we have
$$T\left (\int_{{\mathbb B}_n} f(z)d\mu (z)\right)=\int_{{\mathbb B}_n} T(f)(z)d\mu (z).$$
\end{thm}
Let 
$0 < p < \infty.$ The Bochner-Lebesgue space $L^{p}_{\alpha}(\mathbb{B}_{n},E)$ consists of all vector-valued measurable functions $f:\mathbb{B}_{n} \rightarrow E$ such that $$\displaystyle \|f\|^p_{L^{p}_{\alpha}(\mathbb{B}_{n},E)} = \int_{\mathbb{B}_n}\|f(z)\|^p_{E}\mathrm{d}\nu_{\alpha}(z) < \infty.$$ 
The Bochner-Lebesgue space $L^{\infty}_{\alpha}(\mathbb{B}_{n},E)$ consists of all vector-valued measurable functions $f:\mathbb{B}_{n} \rightarrow E$ such that $$\displaystyle \|f\|_{L^{\infty}_{\alpha}(\mathbb{B}_{n},E)} := esssup\lbrace \|f(z)\|_E : z\in \mathbb B_n \rbrace < \infty.$$

The vector-valued Bergman space $A^{p}_{\alpha}(\mathbb{B}_{n},E)$ is the subspace of $L^{p}_{\alpha}(\mathbb{B}_{n},E)$ consisting of $E-$valued holomorphic functions. When $1 \leq p < \infty,$ the vector-valued Bergman space $A^{p}_{\alpha}(\mathbb{B}_{n},E)$ is a Banach space with the $L^{p}_{\alpha}(\mathbb{B}_{n},E)$ norm. 
We denote by $H^\infty (\mathbb{B}_{n},E)$ the space of bounded $E-$valued holomorphic functions on ${\mathbb B}_n.$ The proof of the following density result can be found in \cite{DT2024} and \cite{D2024}.
\begin{prop}\label{pro12}
For all $1 \leq p < \infty,$ the subspace $H^\infty (\mathbb{B}_{n},E)$ is dense in the vector-valued Bergman space $A^{p}_{\alpha}(\mathbb{B}_{n},E).$
\end{prop}

For $E = \mathbb{C}$ and $p = 2,$ it is well-known that the scalar-valued Bergman space $A^{2}_{\alpha}(\mathbb{B}_{n},\mathbb{C})$ is a reproducing kernel Hilbert space. Let $w \in \mathbb{B}_{n};$ the point evaluation reproducing kernel $K^{\alpha}_{w}$ of $A^{2}_{\alpha}(\mathbb{B}_{n},\mathbb{C})$ is defined by:
$$K_{w}^{\alpha}(z):= \dfrac{1}{(1 - \langle z,w \rangle)^{n+1+\alpha}},\hspace{1cm} z \in \mathbb{B}_{n}.$$ The associated normalized  kernel $k^{\alpha}_{w}$ of $A^{2}_{\alpha}(\mathbb{B}_{n},\mathbb{C})$ is given by
$$k^{\alpha}_{w}(z) = \dfrac{(1 - |w|^2)^{\frac{n+1+\alpha}{2}}}{(1 - \langle z,w \rangle)^{n+1+\alpha}},\hspace{1cm} z \in \mathbb{B}_{n}.$$ 
The orthogonal projector $P_{\alpha}$ of $L^{2}_{\alpha}(\mathbb{B}_{n},\mathbb{C})$ onto $A^{2}_{\alpha}(\mathbb{B}_{n},\mathbb{C})$ is called the Bergman projector. It is well known that for $f \in L^{2}_{\alpha}(\mathbb{B}_{n},\mathbb{C})$ and $z \in \mathbb{B}_{n},$ we have 
$$P_{\alpha}f(z): = \int_{\mathbb{B}_n}\overline{K^{\alpha}_{z}(w)}f(w)d\nu_{\alpha}(w).$$
This definition of $P_\alpha$ extends naturally to $L^{2}_{\alpha}(\mathbb{B}_{n}, E)$ and one has the following reproducing property:
$$P_\alpha f=f \quad (f\in A^{2}_{\alpha}(\mathbb{B}_{n}, E)).$$

For $a \in \mathbb{B}_n\setminus \{0\},$ let $\varphi_{a}$ be the involutive automorphism of $\mathbb{B}_{n}$ defined by
$$ 
\varphi_{a}(z) = \dfrac{a - P_{a}(z) - s_{a}Q_{a}(z)}{1 - \langle z,a \rangle}, 
\hspace{1cm} z \in \mathbb{B}_{n},$$
where $P_{a}(z) = \dfrac{\langle z,a \rangle}{|a|^2}a$ is the orthogonal projector from $\mathbb C^n$ onto the subspace of $\mathbb C^n$ spanned by $a,$ $Q_{a} = I - P_{a}$ and $s_{a} = \sqrt{1 - |a|^2}.$ When $a=0,$ we simply define $\varphi_{a} (z)=-z.$ We recall the following definition.
\begin{defn}\label{def13}
Let $E$ be a complex Banach space and $b\in L^1_\alpha ({\mathbb B}_n, E).$
\begin{enumerate}
\item
The Berezin transform of $b$ is the mapping from ${\mathbb B}_n$ to $E$ defined by
$$\widetilde b(z)=\int_{{\mathbb B}_n} \left \vert k_z^\alpha (w)\right \vert^2 b(w){d\nu_{\alpha} (w)} \quad (z\in {\mathbb B}_n).$$
\item
The function $b$ belongs to the space $BMO^1_{\alpha} ({\mathbb B}_n, E)$ if
$$\left \Vert b\right \Vert_{BMO^1_{\alpha} ({\mathbb B}_n, E)}:=\sup \limits_{z\in {\mathbb B}_n} \left \Vert b\circ \varphi_z-\widetilde b(z)\right \Vert_{L^1_\alpha ({\mathbb B}_n, E)}\textless \infty.$$
\end{enumerate}
\end{defn}


We recall the definition of a Toeplitz operator on a $E-$valued Bergman space.

\begin{defn}\label{def14}
Let $E$ and $F$ be two complex Banach spaces. Let $b\in L^1_\alpha ({\mathbb B}_n, \mathcal L(E, F)).$ The Toeplitz operator $T_b$ is densely defined on $A^p_\alpha \left (\mathbb B_n, E\right )$ by
$$T_b f(z)=\int_{{\mathbb B}_n} b(w)(f(w))\overline{K_z^\alpha(w)}d\nu_\alpha (w)$$
for every $f\in H^\infty (\mathbb B_n, E).$
\end{defn}

The following theorem is  a result from \cite{D2024} and \cite{DT2024}.
\begin{thm}\label{thm15}
Let $E$ and $F$ be two complex Banach spaces. Let $b\in BMO^1_{\alpha} ({\mathbb B}_n, \mathcal L(E, F)).$ The following two assertions are equivalent:
\begin{enumerate}
\item
The Toeplitz operator $T_b$ with symbol $b$ extends to a continuous operator from the Bergman space $A^2_\alpha ({\mathbb B}_n, E)$ to the Bergman space $A^2_\alpha ({\mathbb B}_n, F);$
\item
The Berezin transform $\widetilde b$ of $b$ belongs to $L^\infty ({\mathbb B}_n, \mathcal L(E, F)),$ i.e. 
$$\sup \limits_{z\in {\mathbb B}_n} \left \Vert \widetilde b(z)\right \Vert_{\mathcal L(E, F)}<\infty.$$
\end{enumerate}
Moreover, we have 
$$\Vert \widetilde b \Vert_{L^\infty ({\mathbb B}_n, \mathcal L(E, F))}\lesssim \left \Vert T_b\right \Vert_{\mathcal L(A^2_\alpha (\mathbb B_n, E), A^2_\alpha (\mathbb B_n, F))} \lesssim \Vert \widetilde b \Vert_{L^\infty ({\mathbb B}_n, \mathcal L(E, F))}+ \Vert b\Vert_{BMO^1_{\alpha} ({\mathbb B}_n, \mathcal L(E, F))}.$$
\end{thm}

\begin{rem}\label{rem16}
We endow the Cartesian product $BMO^1_{\alpha} ({\mathbb B}_n, \mathcal L(E, F))\times \mathcal L(A^2_\alpha (\mathbb B_n, E), A^2_\alpha (\mathbb B_n, F))$ with the norm
$$\left \Vert (b, T)\right \Vert_{BMO^1_{\alpha} ({\mathbb B}_n, \mathcal L(E, F))\times \mathcal L(A^2_\alpha (\mathbb B_n, E), A^2_\alpha (\mathbb B_n, F))} = \left \Vert b\right \Vert_{BMO^1_{\alpha} ({\mathbb B}_n, \mathcal L(E, F))}+\left \Vert T\right \Vert_{\mathcal L(A^2_\alpha (\mathbb B_n, E), A^2_\alpha (\mathbb B_n, F))}.$$
It is easy to show that the norms
$$\left \Vert (b, T_b)\right \Vert_{BMO^1_{\alpha} ({\mathbb B}_n, \mathcal L(E, F))\times \mathcal L(A^2_\alpha (\mathbb B_n, E), A^2_\alpha (\mathbb B_n, F))} \quad \rm {and} \quad \left \Vert \widetilde {\left \Vert b(\cdot)\right \Vert}_{\mathcal L (E, F)}\right \Vert_{L^\infty (\mathbb B_n, \mathbb C)}$$ are equivalent.
\end{rem}




We call $E^\star$ the topological dual space of the Banach space $E.$ It was proved by Arragui and Blasco \cite{AB2003} (cf. also \cite[Remark 2.2.23]{D2024}) that the topological dual space of $A^2_\alpha (\mathbb B_n, E)$ is identified with the space $A^2_\alpha (\mathbb B_n, E^\star)$ under the pairing
$$\langle f, g\rangle_{\alpha, E}=\int_{\mathbb B_n} \langle f(z), g(z)\rangle_{E, E^\star}d\nu_\alpha (z).$$
We next define the Berezin transform of an operator $S\in \mathcal L(A^2_\alpha (\mathbb B_n, E)).$

\begin{defn}\label{def17}
Let $E, F$ be two Banach spaces and let $S\in \mathcal L(A^2_\alpha (\mathbb B_n, E), A^2_\alpha (\mathbb B_n, F)).$ The Berezin transform $\widetilde S$ of $S$ is the mapping from $\mathbb B_n$ into $\mathcal L(E, F)$ defined by
$$\langle \widetilde S (z) (x), y^\star\rangle_{F, F^\star}:=\langle S(k_z^\alpha x), k_z^\alpha y^\star\rangle_{\alpha, F},$$
for all $x\in E$ and $y^\star \in F^\star.$
\end{defn}

It is easy to show that  (cf. e.g. \cite[Remark 5.4.3]{D2024}) that for every $b\in L^1_\alpha (\mathbb B_n, \mathcal L(E))$ such that $\widetilde {\left \Vert b(\cdot)\right \Vert_{\mathcal L(E)}}\in L^\infty ({\mathbb B}_n, \mathbb C),$ one has the identity $\widetilde{T_b}=\widetilde b.$

We next recall the following definition. 

\begin{defn}\label{def18}
Let $E$ be a Banach space. For every $z\in \mathbb B_n,$ a unitary operator $U_z$ is defined on $A^2_\alpha (\mathbb B_n, E)$ by 
$$\left (U_z f\right )(w)=\left (f\circ \phi_z\right )(w)k_z^\alpha (w) \quad (f\in A^2_\alpha (\mathbb B_n, E), w\in \mathbb B_n).$$
For every operator $S\in \mathcal L(A^2_\alpha (\mathbb B_n, E)),$ we define the operator $S^z \in \mathcal L(A^2_\alpha (\mathbb B_n, E))$ by
$$S^z:=U_zSU_z.$$ 
\end{defn}

In the sequel, on the one hand, we take $E=F=\ell^2$ and we fix  an orthonormal basis $\{e_j\}_{j=1}^\infty$ of the separable Hilbert space $\ell^2.$ We may also take $E=F=\mathbb C,$ in which case, $\mathcal L(\mathbb C)\simeq \mathbb C.$ On the other hand, we shall denote by $d$ a positive integer and we may also take $E=F=\mathbb C^d.$ We mean by dimension the dimension of $E=F:$ in particular, the dimension is one for $\mathbb C, \hskip 2 truemm d$ for $\mathbb C^d$ and infinity for $\ell^2.$

We next record our main results. For the first one, we state the following definition.

\begin{defn}\label{def19}
The Toeplitz algebra $\mathcal T^{d}_\alpha$ is the operator-norm closure in $\mathcal L(A^2_\alpha (\mathbb B_n, \mathbb C^d))$ of the set of finite sums of finite products of Toeplitz operators with $L^\infty \left (\mathbb B_n, \mathcal L(\mathbb C^d)\right )$ symbols.
\end{defn}

Our first result is a generalization of Corollary 3.6 of Sadeghi-Zorboska \cite{SZ2020}.

\begin{thm}\label{thm110}
Let $\beta \in BMO^1_{\alpha} (\mathbb B_n, \mathcal L(\mathbb C^d))$ be such that $\widetilde \beta \in L^\infty (\mathbb B_n, \mathcal L(\mathbb C^d)).$ Then the Toeplitz operator $T_{\beta}$ belongs to the $d-$dimensional Toeplitz algebra $\mathcal T^{d}_\alpha.$ 
\end{thm}

We rely  on the following theorem due to \cite{R2014,RW2015}.

\begin{thm}\label{thm111}
Let $S\in \mathcal L(A^2_\alpha (\mathbb B_n, \mathbb C^d))$ be a member of $\mathcal T^{d}_\alpha.$ The following three assertions are equivalent.
\begin{enumerate}
\item
$S$ is compact;
\item For every $e\in \mathbb C^d,$
$\left \Vert S^z\left (e \right )\right \Vert_{A^2_\alpha({\mathbb B}_n, \mathbb C^d)}\rightarrow 0$ as $z\rightarrow \partial {\mathbb B}_n;$
\item
$\lim \limits_{|z|\rightarrow 1} \left \Vert\widetilde S(z)\right \Vert_{\mathcal L\left (\mathbb C^d\right )}=0.$ 
\end{enumerate}
\end{thm}

Combining Theorem \ref{thm110} and Theorem \ref{thm111}, it is easy to deduce the following corollary.

\begin{cor}\label{cor112}
Let $d$ be a positive integer. Let $\beta \in BMO^1_{\alpha} ({\mathbb B}_n, \mathcal L(\mathbb C^d))$ be such that $\widetilde \beta \in L^\infty ({\mathbb B}_n, \mathbb C^d)).$ Then the following three assertions are equivalent.
\begin{enumerate}
\item
The Toeplitz operator $T_\beta$ is compact on $A^2_\alpha (\mathbb B_n, \mathbb C^d);$
\item For every $e\in \mathbb C^d,$
$\left \Vert T_{\beta}\left (k_z^\alpha e \right )\right \Vert_{A^2_\alpha({\mathbb B}_n, \mathbb C^d)}\rightarrow 0$ as $z\rightarrow \partial {\mathbb B}_n;$
\item 
$\lim \limits_{|z|\rightarrow 1} \left \Vert\widetilde{\beta} (z)\right \Vert_{\mathcal L\left (\mathbb C^d\right )}=0.$ 
\end{enumerate}
\end{cor}

We also develop a method of reduction to a finite dimension. Let $d$ be a positive integer. We call $M_{I_{(d)}}$ the linear operator defined on $\ell^2$ by
$$M_{I_{(d)}} (e)=\sum_{j=d+1}^\infty \langle e, e_j\rangle_{\ell^2} e_j.$$
This operator extends to the linear operator defined on $A^2_\alpha (\mathbb B_n,  \ell^2)$ as
$$M_{I_{(d)}} (f)(z)=\sum_{j=d+1}^\infty \langle f(z), e_j\rangle_{\ell^2} e_j.$$
We have:
$$T_b M_{I_{(d)}}f(z)=T_{b_{(d)}},$$ 
where the symbol $b_{(d)}: \mathbb B_n\rightarrow \mathcal L(\ell^2)$ is defined as follows:
$$
b_{(d)}(w)(e_i)=\left \{
\begin{array}{clcr}
0&\rm if&1\leq i\leq d\\
b(w)(e_i)&\rm if&i\geq d+1 
\end{array}
\right.
.
$$
Our method of reduction allows us to approximate $T_b$ by the operator $M_{I^{(d)}}T_bM_{I^{(d)}},$ where  $M_{I^{(d)}}$ is the linear operator defined on $\ell^2$ by
$$M_{I^{(d)}} (e)=\sum_{j=1}^d \langle e, e_j\rangle_{\ell^2} e_j,$$
which extends to the linear operator defined on $A^2_\alpha (\mathbb B_n,  \ell^2)$ as
$$M_{I^{(d)}} (f)(z)=\sum_{j=1}^d \langle f(z), e_j\rangle_{\ell^2} e_j.$$
We easily point out that the sum $M_{I^{(d)}}+M_{I_{(d)}}$ is the identity operator on $\ell^2$ and on $A^2_\alpha (\mathbb B_n,  \ell^2).$
The operator $M_{I^{(d)}}T_bM_{I^{(d)}}$ is easily identified with an operator $T_{b^{\llcorner (d)\lrcorner}}$ of $\mathcal L(A_\alpha^2 (\mathbb B_n, \mathbb C^d))$ so that we can apply to it Corollary \ref{cor112}. This leads us to the following theorem.

\begin{thm}\label{thm113}
Let $b\in L^1_\alpha (\mathbb B_n, \mathcal L(\ell^2)).$ The associated Toeplitz operator $T_b$ is compact on $A^2_\alpha (\mathbb B_n,  \ell^2)$ if the following four conditions are satisfied.
\begin{enumerate}
\item[(i)]
$b\in BMO^1_{\alpha} ({\mathbb B}_n, \mathcal L(\ell^2))$ and $\widetilde b\in L^\infty ({\mathbb B}_n, \mathcal L(\ell^2)).$
\item[(ii)]
We have
$$\lim \limits_{d\rightarrow \infty} \left \Vert \widetilde{\left \Vert b_{(d)}(\cdot)\right \Vert_{\mathcal L(\ell^2)}}\right \Vert_{L^\infty ({\mathbb B}_n, \mathbb C)}=0;$$

\item[(iii)] 
We have
$$\lim \limits_{d\rightarrow \infty} \left \Vert \widetilde{\left \Vert {\left ((b(\cdot))^\star_{\mathcal L(\ell^2)}\right )_{(d)}}\right \Vert_{\mathcal L(\ell^2)}}\right \Vert_{L^\infty ({\mathbb B}_n, \mathbb C)}=0;$$
\item[(iv)]
For every positive integer $d,$  the following estimate holds.
$$
\sup \limits_{e\in \mathbb C^d, \left \Vert e\right \Vert_{\mathbb C^d}=1} \lim \limits_{z\rightarrow \partial {\mathbb B}_n} \left \Vert T_{b^{\llcorner (d)\lrcorner}}\left (k_z^\alpha e \right )\right \Vert_{A^2_\alpha ({\mathbb B}_n, \mathbb C^d)}=0.
$$
\end{enumerate}
\end{thm}

To state the fourth result, we introduce the following definition drawn from Rahm \cite{R2016}.
\begin{defn}\label{def114}
\begin{enumerate}
\item
A function $u:\mathbb B_n \rightarrow \mathcal L(\ell^2)$ is in $L^\infty_{fin}$ if $\langle u(\cdot)e_j, e_k\rangle_{\ell^2}$ is non-zero for only finitely many indices $(j, k),$ and each function $z\mapsto \langle u(z)e_j, e_k\rangle_{\ell^2}$ is a bounded function on $\mathbb B_n.$
\item
The Toeplitz algebra $\mathcal T_{L^\infty_{fin}}$ is defined as
$$\mathcal T_{L^\infty_{fin}}:=clos_{\mathcal L(A^2_\alpha (\mathbb B_n, \ell^2))}\left \{\sum_{l=1}^L \prod_{j=1}^J T_{u_{j, l}}: u_{j, l}\in L^\infty_{fin}, J, L\textless \infty\right \},$$
where the closure is taken in the operator norm topology on $\mathcal L(A^2_\alpha (\mathbb B_n, \ell^2)).$
\end{enumerate}
\end{defn}

The following result was proved in \cite[Corollary 3.5]{R2016}: 
\begin{thm}\label{thm115}
Let the linear operator $S\in \mathcal L(A^2_\alpha (\mathbb B_n, \ell^2))$ be in the Toeplitz algebra $\mathcal T_{L^\infty_{fin}}.$ If $
\sup \limits_{e\in \mathbb C^d, \left \Vert e\right \Vert_{\mathbb C^d}=1} \lim \limits_{z\rightarrow \partial {\mathbb B}_n} \left \Vert S\left (k_z^\alpha e \right )\right \Vert_{A^2_\alpha({\mathbb B}_n, \ell^2)}=0$ for every positive integer $d,$ then $S$ is compact.
\end{thm}

We prove the following proposition.
\begin{prop}\label{pro116}
Let $b\in BMO^1_{\alpha} ({\mathbb B}_n, \mathcal L(\ell^2))$ be such that $\widetilde b\in L^\infty ({\mathbb B}_n, \mathcal L(\ell^2)).$ For every positive integer $d,$ the operator $M_{I^{(d)}}T_b M_{I^{(d)}}$ belongs to $\mathcal T_{L^\infty_{fin}}.$ Furthermore, if 
$$\lim \limits_{d\rightarrow \infty} \left \Vert T_bM_{I_{(d)}} \right \Vert_{\mathcal L(A^2_\alpha (\mathbb B_n,  \ell^2))}=\lim \limits_{d\rightarrow \infty} \left \Vert M_{I_{(d)}}T_b \right \Vert_{\mathcal L(A^2_\alpha (\mathbb B_n, \ell^2))}=0,$$
then the operator $T_b$ belongs to $\mathcal T_{L^\infty_{fin}}.$
\end{prop}

To state our last results, we need the following definition  drawn from \cite{R2016}. 

\begin{defn}\label{def117}
An operator $S$ on $A^2_\alpha (\mathbb B_n,  \ell^2)$ is said to be {\it sufficiently localized} if it satisfies the following two conditions
\begin{equation}\label{3.2}
\sup \limits_{j=1, 2, \cdots} \hskip 1truemm \sup \limits_{z\in {\mathbb B}_n} \left \{\int_{{\mathbb B}_n} \left (\sum_{i=1}^\infty \left \vert \langle (S^z e_j)(u), e_i\rangle_{\ell^2}\right \vert \right )^p{d\nu}_\alpha (u)   \right \}^{\frac 1p}<\infty
\end{equation}
and
\begin{equation}\label{3.1}
\sup \limits_{j=1, 2, \cdots} \hskip 1truemm \sup \limits_{z\in {\mathbb B}_n} \left \{\int_{{\mathbb B}_n} \left (\sum_{i=1}^\infty \left \vert \langle (\left (S^\star\right )^z e_j)(u), e_i\rangle_{\ell^2}\right \vert \right )^p{d\nu}_\alpha (u)   \right \}^{\frac 1p}<\infty
\end{equation}
for some real number $p\textgreater \frac {n+2+2\alpha}{1+\alpha}.$ In (\ref{3.1}), $S^\star$ denotes the adjoint operator of $S.$
\end{defn}

Rahm \cite[Theorem 3.4]{R2016} established the following theorem. The notation $\left \Vert S\right \Vert_e$ is for the essential norm of a bounded linear operator $S.$

\begin{thm}\label{thm118}
Let $S:A^2_\alpha (\mathbb B_n,  \ell^2)\rightarrow A^2_\alpha (\mathbb B_n,  \ell^2)$ be a sufficiently localized linear operator. Suppose further that \hskip 1truemm $\limsup \limits_{d\rightarrow \infty} \left \Vert SM_{I_{(d)}}\right \Vert_{\mathcal L\left (A^2_\alpha (\mathbb B_n,  \ell^2)\right )}=0.$ Then:
\begin{enumerate}
	\item[(a)]
	$$\left \Vert S\right \Vert_e\simeq \sup \limits_{f\in A^2_\alpha (\mathbb B_n,  \ell^2): \left \Vert f\right \Vert_{A^2_\alpha (\mathbb B_n,  \ell^2)}=1} \limsup \limits_{z\rightarrow \partial \mathbb B_n} \left \Vert S^z f\right \Vert_{A^2_\alpha (\mathbb B_n,  \ell^2)};$$
	\item[(b)]
	If $\sup \limits_{e\in \mathbb C^d, \left \Vert e\right \Vert_{\mathbb C^d}=1} \lim \limits_{z\rightarrow \partial {\mathbb B}_n} \left \Vert S^z e\right \Vert_{A^2_\alpha (\mathbb B_n,  \ell^2)}=0$ for every positive integer $d,$ then $S$ must be compact.
\end{enumerate} 
\end{thm}

Our last but one result is the following theorem.

\begin{thm}\label{thm119}
Let $b \in BMO^1_\alpha ({\mathbb B}_n, \mathcal L(\ell^2, \ell^1\cap \ell^2))$ be such that $\widetilde b \in L^\infty ({\mathbb B}_n, \mathcal L(\ell^2))$ and $\left (b(\cdot)\right )^\star_{\mathcal L(\ell^2)}\in BMO^1_\alpha ({\mathbb B}_n, \mathcal L(\ell^2, \ell^1)).$ Then the Toeplitz operator $T_b$ is sufficiently localized.
\end{thm}

Next, it is easy to deduce the following corollary from Theorem \ref{thm118} and Theorem \ref{thm119}.

\begin{cor}\label{cor120}
Let $b \in BMO^1_\alpha ({\mathbb B}_n, \mathcal L(\ell^2, \ell^1\cap \ell^2))$ be such that $\widetilde b \in L^\infty ({\mathbb B}_n, \mathcal L(\ell^2))$ and $\left (b(\cdot)\right )^\star_{\mathcal L(\ell^2)}\in BMO^1_\alpha ({\mathbb B}_n, \mathcal L(\ell^2, \ell^1)).$ Then the following assertions hold.
\begin{enumerate}
\item
$T_b$ is a bounded operator on $A^2_\alpha (\mathbb B_n,  \ell^2).$
\item
Assume further that 
\begin{equation}\label{infty}
\lim \limits_{d\rightarrow \infty} \left \Vert \widetilde{\left \Vert b_{(d)}(\cdot)\right \Vert_{\mathcal L(\ell^2)}}\right \Vert_{L^\infty ({\mathbb B}_n, \mathbb C)}=0.
\end{equation} 
If $\sup \limits_{e\in \mathbb C^d, \left \Vert e\right \Vert_{\mathbb C^d}=1} \lim \limits_{z\rightarrow \partial {\mathbb B}_n} \left \Vert T_b\left (k_z^\alpha e \right )\right \Vert_{A^2_\alpha (\mathbb B_n,  \ell^2)}=0$ for every positive integer $d$ and for every $e\in \ell^2,$  then $T_b$ must be compact. 
\end{enumerate}
\end{cor}

The plan of this paper is as follows. In Section \ref{sec2}, we review some elementary properties of the space $BMO^1_\alpha ({\mathbb B}_n, \mathcal L(\ell^2))$ and next, we study in detail the class of non compact Toeplitz operators on $A^2_\alpha (\mathbb B_n,  \ell^2)$ introduced in \cite{R2016}. In Section \ref{sec3}, we state three preliminary lemmas. In Section \ref{sec4}, we study the boundedness and the essential norm of Toeplitz operators with symbols in $BMO^1_\alpha ({\mathbb B}_n, \mathcal L(\ell^2)).$ In Section \ref{sec5}, we describe our method of reduction to the finite dimension. In Section \ref{sec6}, we exhibit a class of examples of  compact Toeplitz operators on $A^2_\alpha (\mathbb B_n, \ell^2).$ In Section \ref{sec7}, we prove Theorem \ref{thm110} and Theorem \ref{thm113}. In between, we go over again the proof of Theorem \ref{thm111}. In Section \ref{sec8}, we prove Proposition \ref{pro116}. In Section \ref{sec9}, we show that Theorem \ref{thm110} does not extend without extra conditions to the infinite dimension. In Section \ref{sec10}, we prove Theorem \ref{thm119}. Section \ref{sec11} revisits the classes of examples of Toeplitz operators studied in Subsection \ref{ssec22} and Section \ref{sec5}. In Section \ref{sec12}, we state a corollary which should have been stated by Rahm at the end of his paper \cite{R2016}. Finally, Section \ref{sec13} just features an open question inspired by \cite{SZ2020} and \cite{X2015}.

		\section{The Space $BMO^1_{\alpha} ({\mathbb B}_n, \mathcal L(\ell^2))$} \label{sec2}
\subsection{Preliminary results}\label{ssec21}
Let $b(\cdot)$ be a matrix function defined on ${\mathbb B}_n,$ with values in $\mathcal L(\ell^2).$ In the sequel, for two positive integers $i$ and $j,$ we shall denote $b_{ij}$ the (scalar) entry function $\langle b(\cdot)(e_i), e_j \rangle_{\ell^2}.$ Moreover, if $b_{ij}\in L^1_\alpha ({\mathbb B}_n, \mathbb C),$ we shall denote by $\widetilde {b_{ij}}$ the (scalar) Berezin transform of $b_{ij}.$
 \begin{lem}\label{lem21}
\begin{enumerate}
\item[1.]
Let $b\in L^1_\alpha ({\mathbb B}_n, \mathcal L(\ell^2)).$ Then the Berezin transform $\widetilde {b_{ij}}$ of the scalar function $b_{ij}$ is given by
\begin{equation}\label{0}
\widetilde {b_{ij}}(z)= \left (\widetilde {b} (z)\right )_{ij}\quad (z\in {\mathbb B}_n).
\end{equation}
\item[2.]
Let $b\in BMO^1_{\alpha} ({\mathbb B}_n, \mathcal L(\ell^2)).$ Then for all positive integers $i, j,$ the scalar function $b_{ij}$  belongs to $BMO^1_{\alpha} ({\mathbb B}_n, \mathbb C).$ Moreover, if $\widetilde b\in L^\infty \left (\mathbb B_n, \mathcal L (\ell^2)\right ),$  then the Berezin transform $\widetilde{b_{ij}}$ of $b_{ij}$ is a bounded scalar function on ${\mathbb B}_n.$ More precisely, we have

\begin{equation}\label{1}
\left \Vert b_{ij}\right \Vert_{BMO^1_{\alpha} ({\mathbb B}_n, \mathbb C)}\leq \left \Vert b\right \Vert_{BMO^1_{\alpha} ({\mathbb B}_n, \mathcal L(\ell^2))};
\end{equation}
and
\begin{equation}\label{3}
\left \Vert \widetilde {b_{ij}}\right \Vert_{L^\infty ({\mathbb B}_n, \mathbb C)}\leq \left \Vert \widetilde b\right \Vert_{L^\infty ({\mathbb B}_n, \mathcal L (\ell^2))}. 
\end{equation}
\end{enumerate}
\end{lem}

\begin{proof}
\begin{enumerate}
\item[1.]
In view of the definition, for every $z\in {\mathbb B}_n,$ we have:
\begin{eqnarray*}
\widetilde {b_{ij}}(z)
&=&\int_{{\mathbb B}_n} \left \vert k_z^\alpha (w)\right \vert^2 b_{ij} (w)d\nu_{\alpha}(w)\\
&=&\int_{{\mathbb B}_n} \left \vert k_z^\alpha (w)\right \vert^2 \langle b(w)(e_i), e_j\rangle_{\ell^2}d\nu_{\alpha}(w)\\
&=&\left\langle \left (\int_{{\mathbb B}_n} \left \vert k_z^\alpha (w)\right \vert^2 b(w)d\nu_{\alpha}(w)\right )(e_i), e_j\right\rangle_{\ell^2}\\
&=&\langle \widetilde b(z)(e_i), e_j \rangle_{\ell^2}\\
&=&\left (\widetilde {b} (z)\right )_{ij}.
\end{eqnarray*}

For the third equality, we used Hille's Theorem (Theorem \ref{thm11}).

\item[2.]
We recall that
$$ \left \Vert b\right \Vert_{BMO^1_{\alpha} ({\mathbb B}_n, \mathcal L(\ell^2))}=\sup \limits_{z\in {\mathbb B}_n} \left \Vert b\circ \varphi_z-\widetilde b(z)\right \Vert_{L^1_\alpha ({\mathbb B}_n, \mathcal L(\ell^2))}$$
and in view of 1.,
\begin{eqnarray*}
\left \Vert b_{ij}\right \Vert_{BMO^1_{\alpha} ({\mathbb B}_n, \mathbb C)}
&=&\sup \limits_{z\in {\mathbb B}_n} \int_{{\mathbb B}_n} \left \vert \left (b_{ij}\circ \varphi_z\right ) (w)-\widetilde {b_{ij}}(z)\right \vert d\nu_{\alpha}(w)\\
&=&\sup \limits_{z\in {\mathbb B}_n} \int_{{\mathbb B}_n} \left \vert \langle \left (b\circ \varphi_z(w)-\widetilde b(z) \right )\left (e_i\right ), e_j
\rangle_{\ell^2}\right \vert d\nu_{\alpha}(w)\\
&=&\sup \limits_{z\in {\mathbb B}_n} \int_{{\mathbb B}_n} \left \vert \left (b\circ \varphi_z(w)-\widetilde b(z)\right )_{ij}\right \vert d\nu_{\alpha}(w).
\end{eqnarray*}

We have

\begin{eqnarray*}
\left \Vert b\circ \varphi_z-\widetilde b(z)\right \Vert_{L^1_\alpha ({\mathbb B}_n, \mathcal L(\ell^2))}
&=&\int_{{\mathbb B}_n} \sup \limits_{\left \Vert e\right \Vert_{\ell^2}=1} \left \Vert \left (b\circ \varphi_z (w)-\widetilde b(z)\right )(e)\right \Vert_{\ell^2}d\nu_{\alpha} (w)\\
&=&\int_{{\mathbb B}_n}\sup \limits_{\left \Vert e\right \Vert_{\ell^2}=1} \left \{\sum_{j=1}^\infty \left \vert \left (\langle b\circ \varphi_z (w)-\widetilde b(z)\right )(e), e_j\rangle_{\ell^2}\right \vert^2\right \}^{\frac 12}d\nu_{\alpha} (w)\\
&\geq& \int_{{\mathbb B}_n} \left \vert \left (\langle b\circ \varphi_z (w)-\widetilde b(z)\right )(e_i), e_j\rangle_{\ell^2}\right \vert d\nu_{\alpha}(w)\\
&=&\int_{{\mathbb B}_n} \left \vert \left (b\circ \varphi_z (w)-\widetilde b(z)\right )_{ij}\right \vert d\nu_{\alpha} (w).
\end{eqnarray*}

Taking the supremum over $z\in {\mathbb B}_n,$ we obtain the estimate (\ref{1}).\\
Let us next show (\ref{3}). 
We have for every $z\in {\mathbb B}_n:$
\begin{eqnarray*}
\left \Vert \widetilde b(z)\right \Vert_{\mathcal L(\ell^2)}
&=&\sup \limits_{\left \Vert e\right \Vert_{\ell^2}=1} \left \Vert \widetilde b(z)(e)\right \Vert_{\ell^2}\\
&=&\sup \limits_{\left \Vert e\right \Vert_{\ell^2}=1} \left \{\sum_{j=1}^\infty \left \vert \langle \widetilde b(z)(e), e_j\rangle_{\ell^2}\right \vert^2\right \}^{\frac 12}\\
&\geq &\left \vert \langle \widetilde b(z)(e_i), e_j\rangle_{\ell^2}\right \vert=\left \vert \left (\widetilde {b}(z)\right )_{ij}\right \vert.
\end{eqnarray*}

Combining with (\ref{0}) and taking the supremum of both sides with respect to $z\in \mathbb B_n,$ the announced conclusion follows. 
\end{enumerate}
\end{proof}


\subsection{A class of examples of non compact Toeplitz operators with $BMO^1_\alpha (\mathbb B_n, \mathcal L(\ell^2))$ operator-valued symbols on $A^2_\alpha (\mathbb B_n, \ell^2)$}\label{ssec22}
The following class of examples was mentioned in the introduction of \cite{R2016}. Let $\tau \in BMO^1_\alpha (\mathbb B_n, \mathbb C)$ be such that $\widetilde \tau \in L^\infty (\mathbb B_n, \mathbb C).$ We call $I$ the identity on $\mathcal L(\ell^2)$ and we write $b:=\tau I.$ 
We first prove the following lemma.

\begin{lem}\label{lem22}
Under the assumptions on $\tau,$ one has $\tau I \in BMO^1_{\alpha} ({\mathbb B}_n, \mathcal L(\ell^2))$ and $\widetilde {\tau I} \in L^\infty ({\mathbb B}_n, \mathcal L(\ell^2)).$ Moreover, for every $z\in {\mathbb B}_n,$ one has $\widetilde {\tau I} (z)=\widetilde \tau (z)I.$
\end{lem}

\begin{proof}
We first show that for every $z\in {\mathbb B}_n,$ one has $\widetilde {\tau I} (z)=\widetilde \tau (z)I.$ For every $e\in \ell^2,$ we have:
\begin{eqnarray}\label{counter}
\widetilde {\tau I} (z)(e)
&= &\int_{{\mathbb B}_n}\left \vert  k_z^\alpha (w)\right \vert^2 (\tau I)(w)(e){d\nu_{\alpha} (w)}\nonumber\\
&=&\int_{{\mathbb B}_n}\left \vert  k_z^\alpha (w)\right  \vert^2 \tau (w)e{d\nu_{\alpha} (w)}\nonumber\\
&=&\left (\int_{{\mathbb B}_n}\left \vert k_z^\alpha (w)\right \vert^2 \tau (w){d\nu_{\alpha}(w)} \right )e \nonumber\\
&=&\widetilde \tau (z)I(e).
\end{eqnarray}
We next show that $\widetilde {\tau I} \in L^\infty ({\mathbb B}_n, \mathcal L(\ell^2)).$  It follows from (\ref{counter}) that
\begin{eqnarray*}
\left \Vert \widetilde {\tau I}\right \Vert_{L^\infty ({\mathbb B}_n, \mathcal L(\ell^2))}
&=&\sup \limits_{z\in {\mathbb B}_n} \sup \limits_{\left \Vert e\right \Vert_{\ell^2}=1} \left \Vert \widetilde {\tau I} (z)(e)\right \Vert_{\ell^2}\\
&=&\sup \limits_{z\in {\mathbb B}_n} \left \vert \widetilde \tau(z)\right \vert\\
&=&\left \Vert \widetilde \tau\right \Vert_{L^\infty ({\mathbb B}_n, \mathbb C)}<\infty.
\end{eqnarray*}

We finally show that $\tau I \in BMO^1_{\alpha} ({\mathbb B}_n, \mathcal L(\ell^2)).$ For every $e\in \ell^2$ such that $\left \Vert e\right \Vert_{\ell^2}=1,$ we have:
$$\left ((\tau I)\circ \varphi_z (w) - \widetilde {\tau I}(z)\right )(e)=\left (\tau\circ \varphi_z (w)-\widetilde \tau (z)\right )e.$$
So 

\begin{eqnarray*}
\left \Vert \tau I\circ \varphi_z (w) - \widetilde {\tau I}(z)\right \Vert_{\mathcal L(\ell^2)}
&=&\sup \limits_{\left \Vert e\right \Vert_\ell^2=1} \left \Vert \left ((\tau I)\circ \varphi_z (w) - \widetilde {\tau I}(z)\right )(e)\right \Vert_{\ell^2}\\
&=&\left \vert \tau\circ \varphi_z (w)-\widetilde \tau (z)\right \vert.
\end{eqnarray*}

The hypothesis $\tau \in BMO^1_{\alpha} ({\mathbb B}_n, \mathbb C)$ implies the announced conclusion.
\end{proof}

From Theorem \ref{thm15} ($E=F=\ell^2$) and Lemma \ref{lem22}, we deduce the following corollary.

\begin{cor}\label{cor23}
Let $\tau \in BMO^1_{\alpha} ({\mathbb B}_n, \mathbb C)$ be such that $\widetilde \tau \in L^\infty ({\mathbb B}_n, \mathbb C).$ Then the Toeplitz operator $T_{\tau I}$ is bounded on $A^2_\alpha (\mathbb B_n,  \ell^2).$
\end{cor}


We next prove the following result.

\begin{prop}\label{pro24}
For the symbol $b=\tau I$ just defined, the Toeplitz operator $T_b$ is not a compact operator on $A^2_\alpha (\mathbb B_n, \ell^2)$ unless the scalar symbol $\tau$ vanishes identically.
\end{prop}

\begin{proof}
We assume that $T_b$ is a compact operator on $A^2_\alpha (\mathbb B_n, \ell^2).$ We recall that for every $g\in A^2_\alpha (\mathbb B_n, \ell^2),$ we have
$$\left \Vert g\right \Vert_{A^2_\alpha (\mathbb B_n, \ell^2)}^2:=\sum_{j=1}^\infty \int_{\mathbb B_n} \left \vert \langle g(z), e_j\rangle_{\ell^2}\right \vert^2d\nu_\alpha (z)\textless \infty.$$

Then $\lim \limits_{j\rightarrow \infty} \int_{\mathbb B_n} \left \vert \langle g(z), e_j\rangle_{\ell^2}\right \vert^2d\nu_\alpha (z)=0.$ \\
Next, for every scalar function $f\in H^\infty (\mathbb B_n, \mathbb C)$  and every positive integer $j,$ we get

\begin{eqnarray*}
\left \vert \langle fe_j, g\rangle_{A^2_\alpha (\mathbb B_n, \ell^2)}\right \vert
&=&\left \vert \int_{\mathbb B_n} \langle f(z)e_j, g(z)\rangle_{\ell^2}d\nu_\alpha (z)\right \vert\\
&=&\left \vert \int_{\mathbb B_n} f(z)\langle e_j, g(z)\rangle_{\ell^2}d\nu_\alpha (z)\right \vert\\
&\leq &\left \Vert f\right \Vert_{H^\infty (\mathbb B_n, \mathbb C)} \int_{\mathbb B_n} \left \vert \langle e_j, g(z)\rangle_{\ell^2}\right \vert d\nu_\alpha (z)\\
&\leq& \left \Vert f\right \Vert_{H^\infty (\mathbb B_n, \mathbb C)} \left (\int_{\mathbb B_n} \left \vert \langle e_j, g(z)\rangle_{\ell^2}\right \vert^2 d\nu_\alpha (z)\right )^{\frac 12},
\end{eqnarray*}
where we used H\"older's inequality in the last line. 
 We conclude that $\lim \limits_{j\rightarrow \infty} \langle fe_j, g\rangle_{A^2_\alpha (\mathbb B_n, \ell^2)}=0.$ In other words, the sequence $\{fe_j\}_{j=1}^\infty$ viewed as a sequence of functions of $A^2_\alpha (\mathbb B_n, \ell^2)$ converges weakly to zero on $A^2_\alpha (\mathbb B_n, \ell^2).$\\
Since $T_b$ is compact, we must have
\begin{equation}\label{comp}
\lim \limits_{j\rightarrow \infty} \left \Vert T_b (fe_j)\right \Vert_{A^2_\alpha (\mathbb B_n, \ell^2)}=0.
\end{equation}
But for every $z\in \mathbb B_n,$ we have
\begin{eqnarray}\label{eq:26}
T_b (fe_j) (z)
&=&\int_{\mathbb B_n} b(w)(f(w)e_j)\overline{K_z^\alpha (w)}d\nu_\alpha (w)\nonumber\\
&=&\left (\int_{\mathbb B_n} \tau (w)f(w)\overline{K_z^\alpha (w)}d\nu_\alpha (w)  \right )e_j \nonumber\\
&=&\left (T_\tau^{\mathbb C} f(z)\right )e_j,
\end{eqnarray}

where $T_\tau^{\mathbb C}$ denotes the scalar Toeplitz operator with scalar symbol $\tau,$
so that 
$$\left \Vert T_b (fe_j)\right \Vert_{A^2_\alpha (\mathbb B_n, \ell^2)}=\left \Vert T_\tau^{\mathbb C} f\right \Vert_{A^2_\alpha (\mathbb B_n, \mathbb C)}$$ for every positive integer $j.$ The equality (\ref{comp}) then gives that $T_\tau^{\mathbb C}=0$ on $A^2_\alpha (\mathbb B_n, \mathbb C)$ because of the density of $H^\infty (\mathbb B_n, \mathbb C)$ in $A^2_\alpha (\mathbb B_n, \mathbb C).$ 
Finally, since $\tau \in BMO^1_\alpha (\mathbb B_n, \mathbb C)$ and  $\widetilde \tau \in L^\infty (\mathbb B_n, \mathbb C),$ in view of Theorem \ref{thm15} with $E=F=\mathbb C,$ we get $\left \Vert \widetilde {\tau (\cdot)}\right \Vert_{L^\infty (\mathbb B_n, \mathbb C)}=0,$ which easily implies that $\tau \equiv 0$ a.e. on $\mathbb B_n$ \cite[Proposition 2.6]{HKZ00}.
\end{proof}

For the particular case $\tau=\chi_{r\mathbb B_n}, 0\leq r \textless 1,$ we deduce the following corollary which differs completely from the one-dimensional case; cf. e.g. \cite{Z2003}.

\begin{cor}\label{cor25}
The Toeplitz operator $T_{\chi_{r\mathbb B_n} I}$ is not compact unless $r=0,$ i.e. its symbol vanishes identically.
\end{cor}


\section{Three preliminary lemmas}\label{sec3}
In this section, we recall two preliminary lemmas proved in \cite{D2024} and \cite{DT2024}. Let $E$ and $F$ be two complex Banach spaces.

\begin{lem}\label{lem31}
Let $b \in L^1 \left (\mathbb B_n, \mathcal L(E, F)\right)$ be such that $\widetilde {\left \Vert b(\cdot)\right \Vert_{\mathcal L(E, F)}}\in L^\infty ({\mathbb B}_n, \mathbb C).$ Then the following estimate holds.
$$\int_{{\mathbb B}_n} \log \left (\frac e{1-\vert w\vert^2}\right )\left \Vert b(w)\right \Vert_{\mathcal L(E, F)}{d\nu}_\alpha (w)<\infty.$$
\end{lem}




We next state the second lemma. 
\begin{lem}\label{lem32}
Let $b:{\mathbb B}_n\rightarrow \mathcal L(E, F)$ be such that $\widetilde {\left \Vert b(\cdot)\right \Vert_{\mathcal L(E, F)}}\in L^\infty ({\mathbb B}_n, \mathbb C).$ Then for all $f$ in the dense space $H^\infty ({\mathbb B}_n, E)$ of $A^2_\alpha (\mathbb B_n, E)$ and $g$ in the dense space $H^\infty ({\mathbb B}_n, F^\star)$ of $A^2_\alpha (\mathbb B_n, F^\star)$ we have:
$$\langle T_b f, g\rangle_{\alpha, F}=\int_{{\mathbb B}_n} \langle b(w)(f(w)), g(w)\rangle_{F, F^\star}d\nu_{\alpha}(w).$$ 
\end{lem}

We next state our third lemma. Its assertion (3) follows from the first assertion of Definition \ref{def13} and Lemma \ref{lem32} and its proof can be found in \cite{R2016}.
\begin{lem}\label{lem33}
Let $b \in BMO^1_{\alpha} ({\mathbb B}_n, \mathcal L(E, F))$ be such that $\widetilde b \in L^\infty ({\mathbb B}_n, \mathcal L(E, F)).$ For every $z\in \mathbb B_n,$ the following three identities hold.
\begin{enumerate}
\item
$\left (T_b\right )^z=T_{b\circ \phi_z};$
\item
$\left \Vert \left (T_b\right )^z\left (e \right )\right \Vert_{A^2_\alpha({\mathbb B}_n, F)}=\left \Vert T_b \left (k_z^\alpha e \right )\right \Vert_{A^2_\alpha({\mathbb B}_n, F)};$
\item
$\widetilde{T_b}=\widetilde b.$
\end{enumerate}
\end{lem}

\section{Toeplitz operators with symbols in  $BMO^1_{\alpha} ({\mathbb B}_n, \mathcal L(\ell^2)).$} \label{sec4} 
Througout this section, $d$ denotes a positive integer. The linear operators $M_{I^{(d)}}$ and $M_{I_{(d)}}$ on $\ell^2$ and on $A^2_\alpha (\mathbb B_n,  \ell^2)$ were defined in the introduction before Theorem \ref{thm113}.

\subsection{A preliminary remark}\label{ssec41}
\begin{rem}\label{rem41}
For every bounded operator $A$ on $A^2_\alpha (\mathbb B_n, \ell^2),$ we denote by $\left \Vert A\right \Vert_e$ the essential norm of $A$ on $A^2_\alpha (\mathbb B_n, \ell^2).$ In particular, for every $b\in BMO^1_{\alpha} \left ({\mathbb B}_n, \mathcal L(\ell^2)\right )$ such that $\widetilde b\in L^\infty ({\mathbb B}_n, \mathcal L(\ell^2)),$ we have the following inequalities:
\begin{eqnarray*}
\left \Vert T_b \right \Vert_e
&\leq & \left \Vert T_b -T_bM_{I^{(d)}}\right \Vert_{\mathcal L (A^2_\alpha (\mathbb B_n, \ell^2))}+\left \Vert T_bM_{I^{(d)}}\right \Vert_e\\
&= & \left \Vert T_b \left (I-M_{I^{(d)}}\right )\right \Vert_{\mathcal L (A^2_\alpha (\mathbb B_n, \ell^2))}+\left \Vert T_bM_{I^{(d)}}\right \Vert_e\\
&\leq & \left \Vert T_bM_{I_{(d)}}\right \Vert_{\mathcal L (A^2_\alpha (\mathbb B_n, \ell^2))}+\left \Vert M_{I^{(d)}}T_bM_{I^{(d)}}\right \Vert_e+\left \Vert T_bM_{I^{(d)}} -M_{I^{(d)}}T_bM_{I^{(d)}}\right \Vert_{\mathcal L(A^2_\alpha (\mathbb B_n, \ell^2))}\\
&\leq & \left \Vert T_bM_{I_{(d)}}\right \Vert_{\mathcal L(A^2_\alpha (\mathbb B_n, \ell^2))}+\left \Vert M_{I^{(d)}}T_bM_{I^{(d)}}\right \Vert_e+\left \Vert M_{I_{(d)}}T_bM_{I^{(d)}}\right \Vert_{\mathcal L(A^2_\alpha (\mathbb B_n, \ell^2))}\\
&\leq & \left \Vert T_bM_{I_{(d)}}\right \Vert_{\mathcal L(A^2_\alpha (\mathbb B_n, \ell^2))}+\left \Vert M_{I^{(d)}}T_bM_{I^{(d)}}\right \Vert_e+\left \Vert M_{I_{(d)}}T_b\right \Vert_{\mathcal L(A^2_\alpha (\mathbb B_n, \ell^2))},
\end{eqnarray*}

since $\left \Vert M_{I^{(d)}}\right \Vert_{\mathcal L(A^2_\alpha (\mathbb B_n, \ell^2))}\leq 1.$ For the second inequality, we deduced from the identity $M_{I^{(d)}}+M_{I_{(d)}}=I,$ where $I$ is the identity on $\ell^2$ and on $A^2_\alpha (\mathbb B_n, \ell^2),$ that
$$ T_b -T_bM_{I^{(d)}}=T_b \left (I-M_{I^{(d)}} \right )=T_bM_{I_{(d)}}.$$
\end{rem}

We deduce the following proposition.

\begin{prop}\label{pro42}
Let $b\in BMO^1_{\alpha} \left ({\mathbb B}_n, \mathcal L(\ell^2)\right )$ be such that $\widetilde b\in L^\infty ({\mathbb B}_n, \mathcal L(\ell^2)).$ Suppose further that
$$\lim \limits_{d\rightarrow \infty} \left \Vert T_bM_{I_{(d)}} \right \Vert_{\mathcal L(A^2_\alpha (\mathbb B_n, \ell^2))}=\lim \limits_{d\rightarrow \infty} \left \Vert M_{I_{(d)}}T_b \right \Vert_{\mathcal L(A^2_\alpha (\mathbb B_n, \ell^2))}=0.$$
The operator $T_b$ is compact on $A^2_\alpha (\mathbb B_n, \ell^2)$ if the operator $M_{I^{(d)}}T_bM_{I^{(d)}}$ is compact on $A^2_\alpha (\mathbb B_n, \ell^2)$ for $d$ large.
\end{prop}

\subsection{A  sufficient condition on the symbol $b$ for the estimate \\ $\lim \limits_{d\rightarrow \infty} \left \Vert T_b M_{I_{(d)}}\right \Vert_{\mathcal L(A^2_\alpha (\mathbb B_n,  \mathcal L(\ell^2)))}=0$}\label{ssec42}
We first prove the following proposition.

\begin{prop}\label{pro43}
Let $b\in L^1_\alpha \left (\mathbb B_n, \mathcal L(\ell^2) \right ).$ Then $T_b M_{I_{(d)}}=T_{b_{(d)}},$ where the symbol $b_{(d)}$ is defined as follows:
\begin{equation}\label{defbd}
b_{(d)}(w)(e_i)=\left \{
\begin{array}{clcr}
0&\rm if&1\leq i\leq d\\
b(w)(e_i)&\rm if&i\geq d+1 
\end{array}
\right..
\end{equation}
\end{prop}

\begin{proof}
For every $f\in H^\infty (\mathbb B_n, \ell^2)$ and every positive integer $d,$ we have:
\begin{eqnarray}\label{bd}
T_b M_{I_{(d)}}f(z)&=&T_b \left (\sum_{i=d+1}^\infty f_ie_i\right )(z) \nonumber\\
&=&\int_{{\mathbb B}_n} b(w)\left (\sum_{i=d+1}^\infty f_i(w)e_i\right )\overline{K_z^\alpha (w)}d\nu_{\alpha}(w) \nonumber\\
&=&\int_{{\mathbb B}_n} \sum_{i=d+1}^\infty f_i(w)b(w)(e_i)\overline{K_z^\alpha (w)} d\nu_{\alpha}(w) \nonumber\\
&=&T_{b_{(d)}} f(z).
\end{eqnarray}
\end{proof}

We next prove the following lemma.

\begin{lem}\label{lem44}
Let $b\in BMO^1_{\alpha} \left ({\mathbb B}_n, \mathcal L(\ell^2)\right )$ be such that $\widetilde b\in L^\infty ({\mathbb B}_n, \mathcal L(\ell^2)).$
The following two assertions hold.
\begin{enumerate}
\item
$b_{(d)}\in BMO^1_{\alpha} ({\mathbb B}_n, \mathcal L(\ell^2));$
\item
$\widetilde {b_{(d)}}\in L^\infty ({\mathbb B}_n, \mathcal L(\ell^2)).$
\end{enumerate}
\end{lem}

\begin{proof}
\begin{enumerate}
\item
For every $z\in {\mathbb B}_n,$ we have:
$$\int_{{\mathbb B}_n} \left \Vert b_{(d)}\circ \varphi_z (w)-\widetilde{b_{(d)}}(z)\right \Vert_{\mathcal L(\ell^2)}d\nu_{\alpha} (w)=
\int_{{\mathbb B}_n} \sup \limits_{\left \Vert e\right \Vert_{\ell^2}=1}\left \Vert \left (b_{(d)}\circ \varphi_z (w)-\widetilde{b_{(d)}}(z)\right )(e)\right \Vert_{\ell^2}d\nu_{\alpha} (w).$$
Notice that
\begin{eqnarray}\label{wt}
\widetilde {b_{(d)}}(z)(e_i)&=&\int_{{\mathbb B}_n} b_{(d)}(w)(e_i)\left \vert k_z^\alpha (w)\right \vert^2d\nu_{\alpha} (w)\nonumber\\
&=&\left \{
\begin{array}{clcr}
0&\rm if &  1\leq i\leq d\\
\int_{{\mathbb B}_n} b(w)(e_i)\left \vert k_z^\alpha (w)\right \vert^2d\nu_{\alpha}(w)&\rm if&i\geq d+1 
\end{array}
\right.
\nonumber\\
&=& \left \{
\begin{array}{clcr}
0&\rm if&1\leq i\leq d\\
\widetilde b(z)(e_i) &\rm if&i\geq d+1 
\end{array}
\right.
\end{eqnarray}
Then, if we denote by $\ell^2_{(d)}$ the subspace of $\ell^2$ consisting of complex sequences whose first $d$ terms are zero, we obtain:

\begin{eqnarray*}
\int_{{\mathbb B}_n} \left \Vert b_{(d)}\circ \varphi_z (w)-\widetilde{b_{(d)}}(z)\right \Vert_{\mathcal L(\ell^2)}d\nu_{\alpha}(w) 
	&= &
\int_{{\mathbb B}_n} \sup \limits_{e\in \ell^2_{(d)}, \left \Vert e\right \Vert_{\ell^2}=1}\left \Vert \left (b\circ \varphi_z (w)-\widetilde{b}(z)\right )(e)\right \Vert_{\ell^2}d\nu_{\alpha}(w)\\
&\leq& \int_{{\mathbb B}_n} \sup \limits_{\left \Vert e\right \Vert_{\ell^2}=1}\left \Vert \left (b\circ \varphi_z (w)-\widetilde{b}(z)\right )(e)\right \Vert_{\ell^2}d\nu_{\alpha}(w)\\
&=&\int_{{\mathbb B}_n} \left \Vert b\circ \varphi_z (w)-\widetilde{b}(z)\right \Vert_{\mathcal L(\ell^2)}d\nu_{\alpha}(w)\\
&\leq& \left \Vert b\right \Vert_{BMO^1_{\alpha} ({\mathbb B}_n, \mathcal L(\ell^2))}<\infty.
\end{eqnarray*}

\item
We also have:

\begin{eqnarray*}
\left \Vert \widetilde {b_{(d)}}(z)\right \Vert_{\mathcal L(\ell^2)}
&= & \sup \limits_{\left \Vert e\right \Vert_{\ell^2}=1}\left \Vert \widetilde {b_{(d)}}(z)(e)\right \Vert_{\ell^2}= \sup \limits_{\left \Vert e\right \Vert_{\ell^2_{(d)}}=1}\left \Vert \widetilde b(z)(e)\right \Vert_{\ell^2}\\
&\leq & \sup \limits_{\left \Vert e\right \Vert_{\ell^2}=1}\left \Vert \widetilde b(z)(e)\right \Vert_{\ell^2}= \left \Vert \widetilde {b}(z)\right \Vert_{\mathcal L(\ell^2)}\\
&\leq &\left \Vert \widetilde {b}\right \Vert_{L^\infty ({\mathbb B}_n, \mathcal L(\ell^2))}<\infty.
\end{eqnarray*}

For the second equality, we used the equation (\ref{wt}).
\end{enumerate}
\end{proof}

In view of (\ref{bd}), Remark \ref{rem16}  and the previous lemma, we reach to the following proposition.

\begin{prop}\label{thm45}
	Let $b\in BMO^1_{\alpha} ({\mathbb B}_n, \mathcal L(\ell^2))$ be such that \hskip 1truemm $\widetilde b\in L^\infty ({\mathbb B}_n, \mathcal L(\ell^2)).$ We assume that
	$$\lim \limits_{d\rightarrow \infty} \left \Vert \widetilde{\left \Vert b_{(d)}(\cdot)\right \Vert}_{\mathcal L(\ell^2)}\right \Vert_{L^\infty ({\mathbb B}_n, \mathbb C)}=0.$$ Then 
$$\lim \limits_{d\rightarrow \infty} \left \Vert T_b M_{I_{(d)}}\right \Vert_{\mathcal L(A^2_\alpha (\mathbb B_n, \ell^2))}=0.$$
\end{prop}


The following lemma will be useful.
\begin{lem}\label{lem46}
For every $z\in \mathbb B_n,$ we have the equality:
\begin{multline*}
	\widetilde{\left \Vert b_{(d)}(\cdot)\right \Vert_{\mathcal L(\ell^2)}}(z)=  \\
	\int_{\mathbb B_n}\left \vert k_z^\alpha (w)\right \vert^2 \sup_{\sum_{i=1}^\infty |\lambda_i|^2=1} \left (\sum_{j=1}^\infty \left \vert \sum_{i=d+1}^\infty \lambda_i\langle b(w)(e_i), e_j\rangle_{\ell^2}\right \vert^2 \right )^{\frac 12}d\nu_\alpha (w).
\end{multline*}
\end{lem}

\begin{proof}
For every $w\in \mathbb B_n,$ we have:
\begin{eqnarray*}
\left \Vert b_{(d)}(w) \right \Vert_{\mathcal L(\ell^2)}
&= &\sup \limits_{\left \Vert e\right \Vert_{\ell^2}=1} \left \Vert b_{(d)}(w)(e)\right \Vert_{\ell^2}\\
&=&\sup \limits_{\left \Vert e\right \Vert_{\ell^2}=1} \left (\sum_{j=1}^\infty \left \vert \langle b_{(d)}(w)(e), e_j\rangle_{\ell^2}\right \vert^2 \right )^{\frac 12}\\
&=&\sup \limits_{\sum_{i=1}^\infty |\lambda_i|^2=1} \left (\sum_{j=1}^\infty \left \vert \sum_{i=1}^\infty \lambda_i\langle b_{(d)}(w)(e_i), e_j\rangle_{\ell^2}\right \vert^2 \right )^{\frac 12}\\
&=&\sup \limits_{\sum_{i=1}^\infty |\lambda_i|^2=1} \left (\sum_{j=1}^\infty \left \vert \sum_{i=d+1}^\infty \lambda_i\langle b(w)(e_i), e_j\rangle_{\ell^2}\right \vert^2 \right )^{\frac 12}.
\end{eqnarray*}

For the latter equality, we used the definition of $b_{(d)}$ given in (\ref{defbd}).
\end{proof}

\subsection{A sufficient condition  on the symbol $b$ for the estimate \\
	$\lim \limits_{d\rightarrow \infty} \left \Vert M_{I_{(d)}}T_b \right \Vert_{\mathcal L(A^2_\alpha (\mathbb B_n, \ell^2))}=0$}\label{ssec43}


It is easy to check that the operator $M_{I_{(d)}}$ is self-adjoint on $\ell^2,$ so it is also self-adjoint as an operator on $A^2_\alpha (\mathbb B_n, \ell^2).$ We assume that $b\in BMO^1_{\alpha} ({\mathbb B}_n, \mathcal L(\ell^2))$ and is such that $\widetilde b\in L^\infty ({\mathbb B}_n, \mathcal L(\ell^2)).$ Hence the adjoint operator of $M_{I_{(d)}}T_b$ on $A^2_\alpha (\mathbb B_n,  \ell^2)$ is $\left (T_b\right )^\star M_{I_{(d)}},$ where $\left (T_b\right )^\star$ denotes the adjoint operator of $T_b.$ For every $w\in {\mathbb B}_n,$ we denote $(b(w))^\star_{\mathcal L(\ell^2)}$ the adjoint operator of $b(w)$ in $\mathcal L(\ell^2).$ 
We first prove the following lemma.
\begin{lem}\label{lem47}
Let $b\in BMO^1_{\alpha} ({\mathbb B}_n, \mathcal L(\ell^2))$ be such that $\widetilde b\in L^\infty ({\mathbb B}_n, \mathcal L(\ell^2)).$ Then
$$\left (T_b\right )^\star=T_{(b(\cdot))^\star_{\mathcal L(\ell^2)}}.$$
Moreover,
for all $e, h\in \ell^2,$ we have
\begin{equation}\label{eqq}
\langle \widetilde {(b(\cdot))^\star_{\mathcal L(\ell^2)}}(z)(e), h\rangle_{\ell^2}=\langle e, \widetilde b(z)(h)\rangle_{\ell^2}
\end{equation}
and $(b(\cdot))^\star_{\mathcal L(\ell^2)}$ belongs to $BMO^1_{\alpha} ({\mathbb B}_n, \mathcal L(\ell^2))$ with the same norm as $b.$
\end{lem}

\begin{proof}
For all $f, g$ in the dense space $H^\infty ({\mathbb B}_n, \ell^2)$ of $A^2_\alpha (\mathbb B_n, \ell^2),$ we have:

\begin{eqnarray*}
\langle f, \left (T_b\right )^\star g\rangle_{A^2_\alpha (\mathbb B_n,  \ell^2)}&=&\langle T_b f, g\rangle_{A^2_\alpha (\mathbb B_n, \ell^2)}.
=\int_{{\mathbb B}_n} \langle T_b f(z), g(z)\rangle_{\ell^2}{d\nu}_\alpha (z).
\end{eqnarray*}

In view of Lemma \ref{lem32}, we get:

\begin{eqnarray*}
\langle f, \left (T_b\right )^\star g\rangle_{A^2_\alpha (\mathbb B_n, \ell^2)}
&=&\int_{{\mathbb B}_n} \langle b(w)(f(w)), g(w)\rangle_{\ell^2}d\nu_{\alpha}(w)\\
&=&\int_{{\mathbb B}_n} \langle f(w), (b(w))^\star_{\mathcal L(\ell^2)}g(w)\rangle_{\ell^2}d\nu_{\alpha}(w)\\
&=&\int_{{\mathbb B}_n} \langle P_\alpha f(w), (b(w))^\star_{\mathcal L(\ell^2)}g(w)\rangle_{\ell^2}d\nu_{\alpha}(w)\\
&=&\langle P_\alpha f, M_{(b(\cdot))^\star_{\mathcal L(\ell^2)}}g\rangle_{L^2_\alpha ({\mathbb B}_n, \ell^2)}=\langle f, P_\alpha M_{(b(\cdot))^\star_{\mathcal L(\ell^2)}}g\rangle_{A^2_\alpha (\mathbb B_n, \ell^2)}\\
&=&\langle f, T_{(b(\cdot))^\star_{\mathcal L(\ell^2)}}g\rangle_{A^2_\alpha (\mathbb B_n, \ell^2)},
\end{eqnarray*}

since the Bergman projector $P_\alpha$ reproduces $f$ and it is self-adjoint on $L^2_\alpha (\mathbb B_n, \ell^2).$ 
An identification gives the announced result:
$$\left (T_b\right )^\star=T_{(b(\cdot))^\star_{\mathcal L(\ell^2)}}.$$
Next, we have

\begin{eqnarray*}
\langle \widetilde {(b(\cdot))^\star_{\mathcal L(\ell^2)}}(z)(e), h\rangle_{\ell^2}
&=&\left \langle \int_{{\mathbb B}_n} \left \vert k_z^\alpha (w)\right \vert^2 (b(w))^\star_{\mathcal L(\ell^2)}(e)d\nu_{\alpha}(w), h\right\rangle_{\ell^2}\\
&=&\int_{{\mathbb B}_n} \left \vert k_z^\alpha (w)\right \vert^2 \langle (b(w))^\star_{\mathcal L(\ell^2)} (e), h\rangle_{\ell^2} d\nu_{\alpha}(w)\\
&=&\int_{{\mathbb B}_n} \left \vert k_z^\alpha (w)\right \vert^2 \langle  e, b(w)(h)\rangle_{\ell^2} d\nu_{\alpha}(w)\\
&=&\langle e, \int_{{\mathbb B}_n} \left \vert k_z^\alpha (w)\right \vert^2 b(w)(h)d\nu_{\alpha} (w)\rangle_{\ell^2}\\
&=&\langle e, \widetilde b(z)(h)\rangle_{\ell^2}. 
\end{eqnarray*}

For the second equality, we applied Hille's Theorem (Theorem \ref{thm11}). We have thus proved (\ref{eqq}).
Moreover, from (\ref{eqq}), we obtain:
\begin{eqnarray*}
\left \Vert (b(\varphi_z (w)))^\star_{\mathcal L(\ell^2)}-\widetilde {(b(\cdot))^\star_{\mathcal L(\ell^2)}}(z)\right \Vert_{\mathcal L(\ell^2)}
&=&\sup \limits_{\Vert e\Vert_{\ell^2}=1=\Vert h\Vert_{\ell^2}} \left \vert \langle e, \left (b(\varphi_z (w))-\widetilde b(z)\right )(h)\rangle_{\ell^2}\right \vert\\
&=&\left \Vert b\circ \varphi_z (\cdot)-\widetilde b(z)\right \Vert_{\mathcal L(\ell^2)}.
\end{eqnarray*}
Hence $\left \Vert (b(\cdot))^\star_{\mathcal L(\ell^2)}\right \Vert_{BMO^1_{\alpha} ({\mathbb B}_n, \mathcal L(\ell^2))}=\left \Vert b\right \Vert_{BMO^1_{\alpha} ({\mathbb B}_n, \mathcal L(\ell^2))}.$
\end{proof}

Replacing the symbol $b$ by $(b(\cdot))^\star_{\mathcal L(\ell^2)}$ in Proposition \ref{thm45}, we reach to the following theorem. 

\begin{prop}\label{thm48}
Let $b\in BMO^1_{\alpha} ({\mathbb B}_n, \mathcal L(\ell^2))$ be such that $\widetilde b\in L^\infty ({\mathbb B}_n, \mathcal L(\ell^2)).$ We assume that   
$$\lim \limits_{d\rightarrow \infty} \left \Vert \widetilde{\left \Vert {\left ((b(\cdot))^\star_{\mathcal L(\ell^2)}\right )_{(d)}}\right \Vert_{\mathcal L(\ell^2)}}\right \Vert_{L^\infty ({\mathbb B}_n, \mathbb C)}=0.$$
Then
$$\lim \limits_{d\rightarrow \infty} \left \Vert M_{I_{(d)}}T_b \right \Vert_{\mathcal L(A^2_\alpha (\mathbb B_n, \ell^2))}=0.$$
\end{prop}

\section{A method of reduction to the finite dimension}\label{sec5}
\subsection*{The subspace $\mathcal S^{(d)}$ of $A^2_\alpha (\mathbb B_n, \ell^2)$}\label{ssec51}
For every positive integer $d,$ we define the subspace $\mathcal S^{(d)}$ of $A^2_\alpha (\mathbb B_n, \ell^2)$ as follows:
$$\mathcal S^{(d)}:=\left \{f\in A^2_\alpha (\mathbb B_n, \ell^2): \langle f, e_j \rangle_{\ell^2}\equiv 0 \quad \forall j\geq d+1   \right \}.$$
We define the mapping 
$$
\begin{array}{lccl}
\mathcal I^{(d)} :&\mathcal S^{(d)}&\rightarrow &A^2_\alpha ({\mathbb B}_n, \mathbb C^d)\\
&f&\mapsto &f^{(d)}
\end{array}
$$
where $f^{(d)}\in A^2_\alpha ({\mathbb B}_n, \mathbb C^d)$ is defined by $\langle f^{(d)}, e_j\rangle_{\mathbb C^d}=\langle f, e_j\rangle_{\ell^2}$ for every integer $j$ such that $1\leq j\leq d.$

The proof of the following lemma is left as an exercise.
\begin{lem}\label{lem51}
\begin{enumerate}
\item
The mapping $\mathcal I^{(d)}$ is an isometric isomorphism with inverse operator $\mathcal J^{(d)}$ defined by
$$
\begin{array}{clcr}
\mathcal J^{(d)}:&A^2_\alpha ({\mathbb B}_n, \mathbb C^d)&\rightarrow &\mathcal S^{(d)}\\
&g=\sum_{i=1}^d g_ie_i&\mapsto &f=\sum_{j=1}^\infty f_je_j,
\end{array}
$$
with $f_j=g_j$ if $1\leq j\leq d$ and $f_j\equiv 0$ otherwise.
\item
The map
$$
\begin{array}{clcr}
\mathcal I^{(d)}\circ M_{I^{(d)}}:&A^2_\alpha (\mathbb B_n, \ell^2)&\rightarrow A^2_\alpha ({\mathbb B}_n, \mathbb C^d)\\
&f&\mapsto (M_{I^{(d)}} f)^{(d)}
\end{array}
$$
is onto.
\end{enumerate}
\end{lem}

Let $d$ be a positive integer. The space $\mathcal L(\mathbb C^d)$ consists of 
$d\times d$ matrices with complex entries. We denote by $L^\infty(\mathbb B_n,\mathcal L(\mathbb C^d))$ the space of functions $\varphi: \mathbb B_n \rightarrow \mathcal L(\mathbb C^d)$ such that the function $z\mapsto \left \Vert \varphi (z) \right \Vert_{\mathcal L(\mathbb C^d)}$ is in $L^\infty \left (\mathbb B_n, \mathbb C\right ).$ Note that it is not particularly important which norm is used on $\mathcal L(\mathbb C^d)$  since $\mathbb C^d$ is finite-dimensional.

\begin{defn}\label{def52}
Let $\varphi :\mathbb B_n\rightarrow \mathcal L\left(\ell^2\right )$ and $d=1, 2, \dots.$ We associate the mapping $\varphi^{\llcorner (d)\lrcorner}:{\mathbb B}_n\rightarrow \mathcal L(\mathbb C^d)$ as follows: for all $1\leq k, j\leq d,$ we have
\begin{equation}\label{starstar}
\langle \varphi^{\llcorner (d)\lrcorner}(w)(e_k), e_j\rangle_{\mathbb C^d}=\langle \varphi(w)(\widetilde{e_k}), \widetilde{e_j}\rangle_{\ell^2},
\end{equation}
where $\widetilde{e_k}=(e_k, 0)$ in $\ell^2.$
We shall abuse the notation $\widetilde{e_k}\equiv e_k.$
\end{defn}

We show the following lemma.

\begin{lem}\label{lem53}
Let $d=1, 2, \dots.$
\begin{enumerate}
\item
For every $b \in L^1_\alpha \left ({\mathbb B}_n, \mathcal L(\ell^2)\right ),$ we have $b^{\llcorner (d)\lrcorner}\in L^1_\alpha \left ({\mathbb B}_n, \mathcal L(\mathbb C^d)\right )$ with the estimate
$$\left \Vert b^{\llcorner (d)\lrcorner}\right \Vert_{L^1_\alpha \left ({\mathbb B}_n, \mathcal L(\mathbb C^d)\right )}\leq \left \Vert b\right \Vert_{L^1_\alpha \left ({\mathbb B}_n, \mathcal L(\ell^2)\right )}.$$

Moreover, we have the identity $\widetilde{b^{\llcorner (d)\lrcorner}}=\left (\widetilde b\right)^{\llcorner (d)\lrcorner}$ and if $\widetilde b \in L^\infty \left (\mathbb B_n, \mathcal L(\ell^2)\right ),$ the following estimate holds
\begin{equation}\label{estd}
\left \Vert \widetilde{b^{\llcorner (d)\lrcorner}} \right \Vert_{L^\infty \left (\mathbb B_n, \mathcal L (\mathbb C^d)\right )}\leq \left \Vert \widetilde b \right \Vert_{L^\infty \left (\mathbb B_n, \mathcal L(\ell^2)\right )}.
\end{equation}
\item
For every $b\in BMO^1_\alpha \left ({\mathbb B}_n, \mathcal L(\ell^2)\right ),$ we have $b^{\llcorner (d)\lrcorner}\in BMO^1_\alpha \left ({\mathbb B}_n, \mathcal L(\mathbb C^d)\right )$ with the estimate
$$\left \Vert b^{\llcorner (d)\lrcorner}\right \Vert_{ BMO^1_\alpha \left ({\mathbb B}_n, \mathcal L(\mathbb C^d)\right )}\leq \left \Vert b\right \Vert_{ BMO^1_\alpha \left ({\mathbb B}_n, \mathcal L(\ell^2)\right )}.$$
\end{enumerate}
\end{lem}

\begin{proof}
\begin{enumerate}
\item
We have:
\begin{eqnarray*}
\left \Vert b^{\llcorner (d)\lrcorner}\right \Vert_{L^1\left ({\mathbb B}_n, \mathcal L(\mathbb C^d)\right )}
&=&\int_{\mathbb B_n} \left \Vert b^{\llcorner (d)\lrcorner})(w)\right \Vert_{\mathcal L(\mathbb C^d)}d\nu_\alpha (w)\\
&=&\int_{\mathbb B_n} \sup \limits_{\left \Vert e\right \Vert_{\mathbb C^d}=\left \Vert h\right \Vert_{\mathbb C^d}=1}
\left \vert \langle b^{\llcorner (d)\lrcorner}(w) (e), h\rangle_{\mathbb C^d}\right \vert d\nu_\alpha (w)\\
&=&\int_{\mathbb B_n} \sup \limits_{\sum_{k=1}^d \vert \lambda_k\vert^2=\sum_{j=1}^d \vert \mu_k\vert^2=1} \left \vert \sum_{k,j=1}^d \lambda_k\overline{\mu_j}
\langle b^{\llcorner (d)\lrcorner}(w) (e_k), e_j\rangle_{\mathbb C^d}\right \vert d\nu_\alpha (w)\\
&\leq& \int_{\mathbb B_n} \sup \limits_{\sum_{k=1}^\infty \vert \lambda_k\vert^2=\sum_{j=1}^\infty \vert \mu_k\vert^2=1} \left \vert\sum_{k,j=1}^\infty \lambda_k\overline{\mu_j} 
\langle b(w)(e_k), e_j\rangle_{\ell^2}\right \vert d\nu_\alpha (w)\\
&=&\int_{\mathbb B_n} \sup \limits_{\left \Vert e\right \Vert_{\ell^2}=\left \Vert h\right \Vert_{\ell^2}=1}
\left \vert \langle b(w)(e), h\rangle_{\ell^2}\right \vert d\nu_\alpha (w)\\
&=&\int_{\mathbb B_n} \left \Vert b(w)\right \Vert_{\mathcal L(\ell^2)}d\nu_\alpha (w)=\left \Vert b\right \Vert_{L^1_\alpha \left ({\mathbb B}_n, \mathcal L(\ell^2)\right )}.
\end{eqnarray*}

For all $k, j\in \{1,\cdots,d\}$ and $z\in \mathbb B_n,$ we have:
\begin{eqnarray*}
\langle \widetilde {b^{\llcorner (d)\lrcorner}}(z)e_k, e_j\rangle)_{\mathbb C^d}
&=&\int_{\mathbb B_n} \vert k_z^\alpha (w)\vert^2 \langle b^{\llcorner (d)\lrcorner}(w)e_k, e_j\rangle_{\mathbb C^d}d\nu_\alpha (w) \\
&=&\int_{\mathbb B_n} \vert k_z^\alpha (w)\vert^2 \langle  b(w)e_k, e_j\rangle_{\ell^2}d\nu_\alpha (w) \\
&=&\langle \widetilde {b}(z)e_k, e_j\rangle)_{\ell^2}.
\end{eqnarray*}

The estimate (\ref{estd}) follows easily.
\item
The assertion (2) follows easily from the assertion (1).
\end{enumerate}
\end{proof}

For $b\in L^1_\alpha \left ({\mathbb B}_n, \mathcal L(\ell^2)\right ),$ we call $T_{b^{\llcorner (d)\lrcorner}}$ the Toeplitz operator with symbol $b^{\llcorner (d)\lrcorner}$ densely defined on $A^2_\alpha \left ({\mathbb B}_n, \mathbb C^d\right ).$ We deduce the following corollary:

\begin{cor}\label{cor54}
Let $b\in BMO^1_{\alpha} \left ({\mathbb B}_n, \mathcal L(\ell^2)\right)$ be such that $\widetilde b\in L^\infty ({\mathbb B}_n, \mathcal L(\ell^2)).$ Then $T_{b^{\llcorner (d)\lrcorner}}$ is bounded on $A^2_\alpha \left ({\mathbb B}_n, \mathbb C^d\right ).$
\end{cor}

For every $f=\sum_{k=1}^\infty f_k e_k\in H^\infty (\mathbb B_n, \ell^2)$ and every $z\in {\mathbb B}_n,$ we have:
\begin{eqnarray*}
M_{I^{(d)}}T_b M_{I^{(d)}} f(z)
&=&\sum_{j=1}^d \langle T_b\left (\sum_{k=1}^d f_ke_k \right )(z), e_j \rangle_{\ell^2} e_j\\
&=&\sum_{j=1}^d \langle \int_{{\mathbb B}_n} b(w)\left (\sum_{k=1}^d f_k(w) e_k \right )\overline{K_z^\alpha (w)}d\nu_{\alpha}(w), e_j \rangle_{\ell^2} e_j\\
&=&\sum_{j=1}^d \sum_{k=1}^d \int_{{\mathbb B}_n} f_k(w)\langle b(w)(e_k), e_j\rangle_{\ell^2}\overline{K_z^\alpha(w)}d\nu_{\alpha}(w)e_j.
\end{eqnarray*}

We next consider the Toeplitz operator $T_{b^{\llcorner (d)\lrcorner}}$ with symbol $b^{\llcorner (d)\lrcorner}$ defined on $A^2_\alpha ({\mathbb B}_n, \mathbb C^d).$  
For every $g=\sum_{k=1}^d g_k e_k\in A^2_\alpha ({\mathbb B}_n, \mathbb C^d),$ we have:
\begin{eqnarray}\label{tbd}
T_{b^{\llcorner (d)\lrcorner}}  g(z)
&=&\sum_{j=1}^d \langle T_{b^{\llcorner (d)\lrcorner}}\left (\sum_{k=1}^d g_ke_k \right )(z), e_j \rangle_{\mathbb C^d} e_j \nonumber\\
&=&\sum_{j=1}^d \langle \int_{{\mathbb B}_n} {b^{\llcorner (d)\lrcorner}}(w)\left (\sum_{k=1}^d g_k(w) e_k \right )\overline{K_z^\alpha (w)}d\nu_{\alpha}(w), e_j \rangle_{\mathbb C^d} e_j \nonumber\\
&=&\sum_{j=1}^d \sum_{k=1}^d\langle \int_{{\mathbb B}_n} g_k(w){b^{\llcorner (d)\lrcorner}}(w)(e_k)\overline{K_z^\alpha(w)}d\nu_{\alpha}(w), e_j \rangle_{\mathbb C^d} e_j \nonumber\\
&=&\sum_{j=1}^d \sum_{k=1}^d \int_{{\mathbb B}_n} g_k(w)\langle b^{\llcorner (d)\lrcorner}(w)(e_k), e_j\rangle_{\mathbb C^d}\overline{K_z^\alpha (w)}d\nu_{\alpha}(w)e_j \nonumber\\
&=&\sum_{j=1}^d \sum_{k=1}^d \int_{{\mathbb B}_n} g_k(w)\langle b(w)(e_k), e_j\rangle_{\ell^2}\overline{K_z^\alpha (w)}d\nu_{\alpha}(w)e_j,
\end{eqnarray}
where the latter equality follows from (\ref{starstar}).
Taking into account the assertion (2) of Lemma \ref{lem51}, we have proved the following lemma.

\begin{lem}\label{lem55}
For every $f\in A^2_\alpha (\mathbb B_n, \ell^2),$ the following equality holds.
\begin{equation}\label{Tb}
M_{I^{(d)}}T_b M_{I^{(d)}}f=\mathcal J^{(d)}T_{b^{\llcorner (d)\lrcorner}}\left (M_{I^{(d)}}f\right )^{(d)}.
\end{equation}
In other words, 
\begin{equation}\label{IdJd}
M_{I^{(d)}}T_b M_{I^{(d)}}=\mathcal J^{(d)}T_{b^{\llcorner (d)\lrcorner}}\mathcal I^{(d)}M_{I^{(d)}}.
\end{equation}
\end{lem}

We next prove the following proposition.

\begin{prop}\label{pro56}
The following two assertions are equivalent.
\begin{enumerate}
\item
$M_{I^{(d)}}T_b M_{I^{(d)}}$ is bounded (resp. compact) on $A^2_\alpha (\mathbb B_n, \ell^2);$
\item
$T_{b^{\llcorner (d)\lrcorner}}$ is bounded (resp. compact) on $A^2_\alpha ({\mathbb B}_n, \mathbb C^d).$
\end{enumerate}
Moreover, the operator norms $\left \Vert M_{I^{(d)}}T_b M_{I^{(d)}}\right \Vert_{\mathcal L(A^2_\alpha (\mathbb B_n, \ell^2))}$ and $\left \Vert T_{b^{\llcorner (d)\lrcorner}}\right \Vert_{\mathcal L\left (A^2_\alpha ({\mathbb B}_n, \mathbb C^d) \right )}$ are equal.
\end{prop}

\begin{proof}
\begin{enumerate}
\item[(i)]
Assume that $M_{I^{(d)}}T_b M_{I^{(d)}}$ is bounded on $A^2_\alpha (\mathbb B_n, \ell^2).$ Notice that on $A^2_\alpha (\mathbb B_n, \ell^2),$ the following identity holds: $M_{I^{(d)}}\circ M_{I^{(d)}}=M_{I^{(d)}}.$  In view of (\ref{Tb}), for every $f\in A^2_\alpha (\mathbb B_n, \ell^2),$ since $\mathcal J^{(d)}$ is an isometry, we have:
\begin{eqnarray*}
\left \Vert T_{b^{\llcorner (d)\lrcorner}}\left (M_{I^{(d)}}f\right )^{(d)}\right \Vert_{A^2_\alpha ({\mathbb B}_n, \mathbb C^d)}
&=&\left \Vert \mathcal J^{(d)}T_{b^{\llcorner (d)\lrcorner}}\left (M_{I^{(d)}}f\right )^{(d)}\right \Vert_{A^2_\alpha (\mathbb B_n, \ell^2)}\\
&=&\left \Vert M_{I^{(d)}}T_b M_{I^{(d)}}f\right \Vert_{A^2_\alpha (\mathbb B_n, \ell^2)}\\
&=&\left \Vert M_{I^{(d)}}T_b M_{I^{(d)}}M_{I^{(d)}}f\right \Vert_{A^2_\alpha (\mathbb B_n, \ell^2)}\\
&\leq& \left \Vert M_{I^{(d)}}T_b M_{I^{(d)}}\right \Vert_{\mathcal L(A^2_\alpha (\mathbb B_n, \ell^2))}\left \Vert M_{I^{(d)}}f\right \Vert_{A^2_\alpha (\mathbb B_n, \ell^2)}\\
&\leq & \left \Vert M_{I^{(d)}}T_b M_{I^{(d)}}\right \Vert_{\mathcal L(A^2_\alpha (\mathbb B_n, \ell^2))}\left \Vert \left (M_{I^{(d)}}f\right )^{(d)}\right \Vert_{A^2_\alpha ({\mathbb B}_n, \mathbb C^d)},
\end{eqnarray*}
by assertion (1) of Lemma \ref{lem51}.
Since by assertion (2) of Lemma \ref{lem51}, the map
$$
\begin{array}{clcr}
A^2_\alpha (\mathbb B_n, \ell^2)&\rightarrow A^2_\alpha ({\mathbb B}_n, \mathbb C^d)\\
f&\mapsto (M_{I^{(d)}} f)^{(d)}
\end{array}
$$
is onto, this implies that for every $g\in A^2_\alpha ({\mathbb B}_n, \mathbb C^d),$ we have:
$$\left \Vert T_{b^{\llcorner (d)\lrcorner}} g\right \Vert_{A^2_\alpha ({\mathbb B}_n, \mathbb C^d)}\leq \left \Vert M_{I^{(d)}}T_b M_{I^{(d)}}\right \Vert_{\mathcal L(A^2_\alpha (\mathbb B_n, \ell^2))}\left \Vert g\right \Vert_{A^2_\alpha ({\mathbb B}_n, \mathbb C^d)}.$$
In other words,
$$\left \Vert T_{b^{\llcorner (d)\lrcorner}}\right \Vert_{\mathcal L(A^2_\alpha ({\mathbb B}_n, \mathbb C^d))}\leq \left \Vert M_{I^{(d)}}T_b M_{I^{(d)}}\right \Vert_{\mathcal L(A^2_\alpha (\mathbb B_n, \ell^2))}.$$
Conversely, 
the announced conclusion then follows from the identity (\ref{IdJd}), since $\mathcal I^{(d)}$ and $\mathcal J^{(d)}$ are isometries and $M_{I^{(d)}}$ is bounded on $A^2_\alpha (\mathbb B_n, \ell^2)$ with norm 1. 
\item[(ii)]
Assume first that the operator $M_{I^{(d)}}T_b M_{I^{(d)}}$ is compact on $A^2_\alpha (\mathbb B_n, \ell^2).$ This means that we have the implication:
\begin{center}
$\{f_m\}\rightarrow 0$ weakly on $A^2_\alpha (\mathbb B_n, \ell^2)  \Rightarrow \left \Vert M_{I^{(d)}}T_b M_{I^{(d)}}f_m\right \Vert_{A^2_\alpha (\mathbb B_n,  \ell^2)}\rightarrow 0 \quad (m\rightarrow \infty).$
\end{center}
To deduce the compactness of $T_{b^{\llcorner (d)\lrcorner}}$ on $A^2_\alpha ({\mathbb B}_n, \mathbb C^d),$ we need to show the following implication:
\begin{center}
$\{g_m\}\rightarrow 0$ weakly on $A^2_\alpha ({\mathbb B}_n, \mathbb C^d) \Rightarrow \left \Vert T_{b^{\llcorner (d)\lrcorner}} g_m\right \Vert_{A^2_\alpha ({\mathbb B}_n, \mathbb C^d)}\rightarrow 0 \quad (m\rightarrow \infty).$
\end{center}
We assume that $\{g_m\}\rightarrow 0$ weakly on $A^2_\alpha({\mathbb B}_n, \mathbb C^d) \quad (m\rightarrow \infty).$ By assertion (2) of Lemma \ref{lem51}, for every   positive integer $m,$ there exists $f_m \in A^2_\alpha (\mathbb B_n,  \ell^2)$ such that $g_m=\left (M_{I^{(d)}}f_m\right )^{(d)}.$  Let us prove that $M_{I^{(d)}} f_m\rightarrow 0$ weakly on $A^2_\alpha (\mathbb B_n,  \ell^2) \quad (m\rightarrow \infty).$ 
For every $\varphi \in A^2_\alpha (\mathbb B_n,  \ell^2),$ we have:

\begin{eqnarray*}
\langle M_{I^{(d)}} f_m, \varphi \rangle_{A^2_\alpha (\mathbb B_n,  \ell^2)}
&=&\int_{{\mathbb B}_n} \langle M_{I^{(d)}} f_m, \varphi \rangle_{\ell^2}{d\nu}_\alpha\\
&=&\int_{{\mathbb B}_n} \langle M_{I^{(d)}} f_m, M_{I^{(d)}}\varphi \rangle_{\ell^2}{d\nu}_\alpha\\
&=&\int_{{\mathbb B}_n} \langle \left (M_{I^{(d)}} f_m\right )^{(d)}, \left (M_{I^{(d)}}\varphi\right )^{(d)} \rangle_{\mathbb C^d}{d\nu}_\alpha\\
&=&\int_{{\mathbb B}_n} \langle g_m, \left (M_{I^{(d)}}\varphi\right )^{(d)} \rangle_{\mathbb C^d}{d\nu}_\alpha\\
&=&\langle g_m, \left (M_{I^{(d)}}\varphi\right )^{(d)} \rangle_{A^2_\alpha ({\mathbb B}_n, \mathbb C^d)} \rightarrow 0 \quad (m\rightarrow \infty).
\end{eqnarray*}

In view of the hypothesis, by the identity (\ref{Tb}), we obtain that
$$\left \Vert T_{b^{\llcorner (d)\lrcorner}} g_m\right \Vert_{A^2_\alpha ({\mathbb B}_n, \mathbb C^d)}=\left \Vert M_{I^{(d)}}T_b M_{I^{(d)}} f_m\right \Vert_{A^2_\alpha (\mathbb B_n,  \ell^2)}\rightarrow 0 \quad (m\rightarrow \infty).$$
We used the fact that $\mathcal I^{(d)}$ is an isometry.\\
Conversely, the compactness of $M_{I^{(d)}}T_b M_{I^{(d)}}$ on $A^2_\alpha (\mathbb B_n,  \ell^2)$ follows from that of $T_{b^{\llcorner (d)\lrcorner}}$ on $A^2_\alpha (\mathbb B_n,  \mathbb C^d)$ in view of the identity (\ref{IdJd}).
\end{enumerate}
\end{proof}

We summarize the results of Section \ref{sec4} and Section \ref{sec5} in the following theorem.

\begin{thm}\label{thm57}
Let $b\in BMO^1_{\alpha} ({\mathbb B}_n, \mathcal L(\ell^2))$ be such that $\widetilde b\in L^\infty ({\mathbb B}_n, \mathcal L(\ell^2)).$ We  assume that
$$\lim \limits_{d\rightarrow \infty} \left \Vert \widetilde{\left \Vert b_{(d)}(\cdot)\right \Vert_{\mathcal L(\ell^2)}}\right \Vert_{L^\infty ({\mathbb B}_n, \mathbb C)}=0$$
and 
$$\lim \limits_{d\rightarrow \infty} \left \Vert \widetilde{\left \Vert {\left ((b(\cdot))^\star_{\mathcal L(\ell^2)}\right )_{(d)}}\right \Vert_{\mathcal L(\ell^2)}}\right \Vert_{L^\infty ({\mathbb B}_n, \mathbb C)}=0.$$
We further assume that $T_{b^{\llcorner (d)\lrcorner}}$ is compact on $A^2_\alpha({\mathbb B}_n, \mathbb C^d)$ for $d$ large. Then $T_b$ is compact on $A^2_\alpha (\mathbb B_n,  \ell^2).$
\end{thm} 

\section{A class of examples of compact Toeplitz operators on $A^2_\alpha (\mathbb B_n, \ell^2)$ with symbols in $BMO^1_{\alpha} ({\mathbb B}_n, \mathcal L(\ell^2))$ }\label{sec6}
Let $\tau \in BMO^1_{\alpha} ({\mathbb B}_n, \mathbb C)$ be such that $\widetilde \tau \in L^\infty ({\mathbb B}_n, \mathbb C).$ We define a function $b\in BMO^1_{\alpha} ({\mathbb B}_n, \mathcal L(\ell^2))$ by
$$b(z)(e_i)=\frac {\tau (z)}{2^i}e_i \quad (z\in {\mathbb B}_n, i=1, 2, \cdots).$$
Associated to $b$ is the diagonal Toeplitz operator $T_b.$ We state the following theorem.

\begin{thm}\label{thm61}
Assume that $\widetilde \tau (z)\rightarrow 0$ as $z\rightarrow \partial {\mathbb B}_n.$ Then the Toeplitz operator $T_b$ is compact on $A^2_\alpha (\mathbb B_n,  \ell^2).$
\end{thm}

\begin{proof}
In view of Theorem \ref{thm57}, we shall check the following four facts:
\begin{enumerate}
\item
$b\in BMO^1_{\alpha} ({\mathbb B}_n, \mathcal L(\ell^2))$ and $\widetilde b\in L^\infty ({\mathbb B}_n, \mathcal L(\ell^2));$
\item
$$\lim \limits_{d\rightarrow \infty} \left \Vert \widetilde{\left \Vert b_{(d)}(\cdot)\right \Vert_{\mathcal L(\ell^2)}}\right \Vert_{L^\infty ({\mathbb B}_n, \mathbb C)}=0;$$
\item 
$$\lim \limits_{d\rightarrow \infty} \left \Vert \widetilde{\left \Vert {\left ((b(\cdot))^\star_{\mathcal L(\ell^2)}\right )_{(d)}}\right \Vert_{\mathcal L(\ell^2)}}\right \Vert_{L^\infty ({\mathbb B}_n, \mathbb C)}=0;$$
\item
$T_{b^{\llcorner (d)\lrcorner}}$ is compact on $A^2_\alpha({\mathbb B}_n, \mathbb C^d)$ for every positive integer $d.$  
\end{enumerate}
\begin{enumerate}
\item
We first show that $\widetilde b\in L^\infty ({\mathbb B}_n, \mathcal L(\ell^2)).$ For every $e=\sum_{i=1}^\infty \lambda_i e_i$ such that $\sum_{i=1}^\infty \left \vert \lambda_i\right \vert^2=1$ and every $z\in {\mathbb B}_n,$ we have:

\begin{eqnarray*}
\widetilde b(z)(e)
&=&\int_{{\mathbb B}_n} \left \vert k_z^\alpha (w)\right \vert^2 b(w)(e)d\nu_{\alpha}(w)\\
&=&\sum_{i=1}^\infty \frac{\lambda_i}{2^i}\left (\int_{{\mathbb B}_n} \left \vert k_z^\alpha (w)\right \vert^2 \tau(w)e_i d\nu_{\alpha}(w) \right )\\
&=&\sum_{i=1}^\infty \frac{\lambda_i}{2^i} \widetilde \tau (z)e_i.
\end{eqnarray*}

So 
\begin{eqnarray*}
\sup \limits_{z\in {\mathbb B}_n}\sup \limits_{\left \Vert e\right \Vert_{\ell^2}=1}\left \Vert \widetilde b(z)(e)\right \Vert_{\ell^2}
&=&\sup \limits_{z\in {\mathbb B}_n}\sup \limits_{\sum_{i=1}^\infty \vert \lambda_i\vert^2=1} \left \vert \widetilde \tau (z)\right \vert \left (\sum_{i=1}^\infty \left (\frac {\vert \lambda_i\vert}{2^i}\right )^2\right )^{\frac 12} \\
&\leq& \sup \limits_{z\in {\mathbb B}_n} \left \vert \widetilde \tau (z)\right \vert.
\end{eqnarray*}

The conclusion follows from the boundedness of the scalar function $\widetilde \tau.$\\
We next show that $b\in BMO^1_{\alpha} ({\mathbb B}_n, \mathcal L(\ell^2)).$ For every $e=\sum_{i=1}^\infty \lambda_i e_i$ such that $\sum_{i=1}^\infty \vert \lambda_i\vert^2=1$ and all $z, w\in {\mathbb B}_n,$ 
we have:
$$\left (b\circ \varphi_z (w)-\widetilde b(z)\right )(e)=\sum_{i=1}^\infty \lambda_i \left (b\circ \varphi_z (w)-\widetilde b(z) \right )\left (e_i\right ).$$
So
\begin{eqnarray*}
\left \Vert b\circ \varphi_z (w)-\widetilde b(z)\right \Vert_{\mathcal L(\ell^2)}
&=&\sup \limits_{\left \Vert e\right \Vert_{\ell^2}=1} \left \Vert \left (b\circ \varphi_z (w)-\widetilde b(z)\right )(e)\right \Vert_{\ell^2}\\
&=&\sup \limits_{\sum_{i=1}^\infty \left \vert \lambda_i\right \vert^2=1}\left \Vert\sum_{i=1}^\infty  \frac {\lambda_i}{2^i} \left (\tau \circ \varphi_z (w)-\widetilde \tau (z) \right )e_i \right \Vert_{\ell^2}\\
&=&\sup \limits_{\sum_{i=1}^\infty \left \vert \lambda_i\right \vert^2=1}\left \{\left \vert \tau \circ \varphi_z (w)-\widetilde \tau (z)\right \vert \left (\sum_{i=1}^\infty \left  (\frac {\vert \lambda_i\vert}{2^i}\right )^2\right )^{\frac 12}\right \}\\
&\leq& \left \vert \tau \circ \varphi_z (w)-\widetilde \tau (z)\right \vert.
\end{eqnarray*}
The hypothesis $\tau \in BMO^1_{\alpha} ({\mathbb B}_n, \mathbb C)$ implies the announced conclusion.
\item
We now prove that 
$$\lim \limits_{d\rightarrow \infty} \left \Vert \widetilde{\left \Vert b_{(d)}(\cdot)\right \Vert_{\mathcal L(\ell^2)}}\right \Vert_{L^\infty ({\mathbb B}_n, \mathbb C)}=0.$$
By Lemma \ref{lem46}, we obtain:
\begin{eqnarray*}
\widetilde{\left \Vert b_{(d)}(\cdot)\right \Vert_{\mathcal L(\ell^2)}}(z)
&=&\int_{\mathbb B_n}\left \vert k_z^\alpha (w)\right \vert^2 \sup \limits_{\sum_{i=1}^\infty |\lambda_i|^2=1} \left (\sum_{j=1}^\infty \left \vert \sum_{i=d+1}^\infty \lambda_i\langle b(w)(e_i), e_j\rangle_{\ell^2}\right \vert^2 \right )^{\frac 12}d\nu_\alpha (w)\\
&=&\int_{\mathbb B_n}\left \vert k_z^\alpha (w)\right \vert^2 \sup \limits_{\sum_{i=1}^\infty |\lambda_i|^2=1} \left (\sum_{j=1}^\infty \left \vert \sum_{i=d+1}^\infty \lambda_i\langle \frac {\tau(w)}{2^i}e_i, e_j\rangle_{\ell^2}\right \vert^2 \right )^{\frac 12}d\nu_\alpha (w)\\
&=&\int_{\mathbb B_n}\left \vert k_z^\alpha (w)\right \vert^2 \sup \limits_{\sum_{i=1}^\infty |\lambda_i|^2=1} \left (\sum_{j=d+1}^\infty \left \vert \lambda_j\frac {\tau(w)}{2^j})\right \vert^2 \right )^{\frac 12}d\nu_\alpha (w)\\
&=&\int_{\mathbb B_n}\left \vert k_z^\alpha (w)\right \vert^2 \left \vert \tau (w)\right \vert d\nu_\alpha (w)\times \sup \limits_{\sum_{i=1}^\infty |\lambda_i|^2=1}\left (\sum_{j=d+1}^\infty \frac {\left \vert \lambda_j\right \vert^2}{2^{2j}} \right )^{\frac 12}\\
&\leq &
\left \Vert \widetilde {\left \vert \tau (\cdot)\right \vert}\right \Vert_{L^\infty (\mathbb B_n, \mathbb C)}\left (\sum_{j=d+1}^\infty \frac 1{2^{2j}} \right )^{\frac 12}.
\end{eqnarray*}

The conclusion follows from the convergence of the series $\sum \frac 1{2^{2j}}$ and the boundedness of $\widetilde {\left \vert \tau (\cdot)\right \vert}.$ The latter follows from the assumptions $\tau \in BMO^1_{\alpha} ({\mathbb B}_n, \mathbb C)$ and $\widetilde \tau \in L^\infty ({\mathbb B}_n, \mathbb C)$ according to Remark \ref{rem16} in the particular case $E=F=\mathbb C.$ 
\item 
We next have to prove that
$$\lim \limits_{d\rightarrow \infty} \left \Vert \widetilde{\left \Vert {\left ((b(\cdot))^\star_{\mathcal L(\ell^2)}\right )_{(d)}}\right \Vert_{\mathcal L(\ell^2)}}\right \Vert_{L^\infty ({\mathbb B}_n, \mathbb C)}=0.$$
The proof is the same as for (2) when $b$ is replaced by $(b(\cdot))^\star_{\mathcal L(\ell^2)}$ 
using the equality 
$$\langle (b(w))^\star_{\mathcal L(\ell^2)}(e_i), e_j\rangle_{\ell^2}=\langle (e_i, b(w)e_j\rangle_{\ell^2}=\frac {\overline{\tau (w)}}{2^j}\langle e_i, e_j\rangle_{\ell^2}.$$

\item 
We finally show that $T_{b^{\llcorner (d)\lrcorner}}$ is compact on $A^2_\alpha({\mathbb B}_n, \mathbb C^d)$ for every positive integer $d.$ For every $g\in A^2_\alpha({\mathbb B}_n, \mathbb C^d),$ it follows from \ref{tbd} that
\begin{eqnarray*}
T_{b^{\llcorner (d)\lrcorner}} g(z)
&=&\sum_{j=1}^d \sum_{k=1}^d \int_{{\mathbb B}_n} \langle g(w), e_k\rangle_{\mathbb C^d}\langle b(w)(e_k), e_j\rangle_{\ell^2}\overline {K_z^\alpha (w)}d\nu_{\alpha}(w)e_j\\
&=&\sum_{j=1}^d \sum_{k=1}^d \int_{{\mathbb B}_n} \langle g(w), e_k\rangle_{\mathbb C^d}\tau(w)\frac 1{2^k}\langle e_k, e_j\rangle_{\mathbb C^d}\overline {K_z^\alpha (w)}d\nu_{\alpha}(w)e_j\\
&=&\sum_{j=1}^d \frac 1{2^j} \int_{{\mathbb B}_n} \langle g(w), e_j\rangle_{\mathbb C^d}\tau (w)\overline {K_z^\alpha (w)}d\nu_{\alpha}(w)e_j\\
&=&\sum_{j=1}^d \frac 1{2^j}T_\tau^{\mathbb C} \left(\langle g(\cdot), e_j\rangle_{\mathbb C^d}\right)(z)e_j.
\end{eqnarray*}

It then suffices to prove that the operator
$$
\begin{array}{ccl}
A^2_\alpha({\mathbb B}_n, \mathbb C^d)&\rightarrow & A^2_\alpha ({\mathbb B}_n, \mathbb C)\\
g&\mapsto & T_\tau^{\mathbb C} \left (\langle g(\cdot), e_j\rangle_{\mathbb C^d}\right )
\end{array}
$$
is compact for every $j=1, 2,\cdots, d.$ Indeed, this operator is the composition of the linear continuous operator
$$
\begin{array}{ccl}
A^2_\alpha({\mathbb B}_n, \mathbb C^d)&\rightarrow & A^2_\alpha ({\mathbb B}_n, \mathbb C)\\
g&\mapsto & \langle g(\cdot), e_j\rangle_{\mathbb C^d}
\end{array}
$$
and the linear compact operator
$$
\begin{array}{ccl}
A^2_\alpha({\mathbb B}_n, \mathbb C)
&\rightarrow & A^2_\alpha ({\mathbb B}_n, \mathbb C)\\
\varphi&\mapsto & T_\tau^{\mathbb C} \varphi.
\end{array}
$$
The latter operator is compact in view of Theorem 1 of Zorboska \cite{Z2003} for $n=1$ and Corollary 4.4 of \cite{AT2010} for $n\geq 2$ both for $\alpha=0.$ Cf. \cite{ZLL2011} for all $\alpha\textgreater -1.$
\end{enumerate}
\end{proof}

\section{Sufficient conditions for the compactness of the Toeplitz operator $T_b$ with a symbol $b$ in $BMO^1_{\alpha} ({\mathbb B}_n, \mathcal L(\mathbb C^d))$}\label{sec7}
Throughout the section, $d$ denotes a fixed positive integer.

\subsection{Proof of Theorem 1.10}\label{ssec71}
The $d-$dimensional Toeplitz algebra $\mathcal T^{d}_\alpha$ was defined in Definition \ref{def19}.

We first recall the following theorem, which is Corollary 3.6 of Sadeghi-Zorboska \cite{SZ2020}.
\begin{thm}\label{thm71}
Let $b \in BMO^1_{\alpha} ({\mathbb B}_n, \mathbb C)$ be such that $\widetilde b \in L^\infty ({\mathbb B}_n, \mathbb C).$ Then the Toeplitz operator $T_{b}$ belongs to the $1-$dimensional Toeplitz algebra $\mathcal T^{1}_\alpha.$ 
\end{thm}

We next prove Theorem \ref{thm110}.
From (\ref{tbd}), we recall that
\begin{eqnarray}\label{Tbd1}
T_{\beta}  g(z)&=&\sum_{j=1}^d \sum_{k=1}^d \int_{{\mathbb B}_n} g_k(w)\beta_{k, j} (w)\overline{K_z^\alpha (w)}d\nu_{\alpha} (w)e_j \nonumber\\
&=&\sum_{j=1}^d \sum_{k=1}^d T^{\mathbb C}_{\beta_{k, j}} g_k(z)e_j.
\end{eqnarray}
For all positive integers $k,j\in \left \{1,\cdots, d\right \},$ we associate with the scalar symbol $\beta_{k, j}$ a new symbol $B_{k, j}: {\mathbb B}_n\rightarrow \mathcal L(\mathbb C^d)$ as follows.
$$
\langle B_{k, j}(w) (e_l), e_m\rangle_{\mathbb C^d}=
\left \{
\begin{array}{clcr}
\beta_{k, j}(w) &\rm if &l=j, \hskip 1truemm m=k\\
0&\rm otherwise&
\end{array}
\right..
$$
So for every $g\in A^2_\alpha({\mathbb B}_n, \mathbb C^d)$ and all positive integers $k, l, m,$ we have:
\begin{eqnarray*}
\langle B_{k, j}(w) (g(w)), e_m\rangle_{\mathbb C^d}
&=&\left\langle B_{k, j}(w) \left(\sum_{l=1}^d g_l(w)e_l\right), e_m\right\rangle_{\mathbb C^d}\\
&=&\sum_{l=1}^d g_l(w)\langle B_{k, j}(w) (e_l), e_m\rangle_{\mathbb C^d}\\
&=&g_j(w)\beta_{k, j}(w)\delta_{m, k}\\
&=&g_j(w)\beta_{k, j}(w)\langle e_k, e_m\rangle_{\mathbb C^d}.
\end{eqnarray*}

Then $B_{k, j}(w) (g(w))=g_j(w)\beta_{k, j}(w)e_k.$
Hence:
\begin{equation}\label{omega}
T_{B_{k, j}} g(z)= \int_{{\mathbb B}_n} B_{k, j}(w) (g(w))\overline{K_z^\alpha(w)}d\nu_{\alpha}(w)=T^{\mathbb C}_{\beta_{k, j}} g_j(z)e_k
\end{equation}
and from (\ref{Tbd1}), we get
$$
T_{\beta}  g(z)=\sum_{j=1}^d \sum_{k=1}^d T_{B_{k, j}} g(z).
$$
It then suffices to show that each Toeplitz operator $T_{B_{k, j}}$ belongs to the $d-$dimensional Toeplitz algebra $\mathcal T^{d}_\alpha.$ In view of Lemma \ref{lem21} and the previous theorem (Corollary 3.6 of Sadeghi-Zorboska \cite{SZ2020}), the scalar Toeplitz operator $T_{\beta_{k, j}}^{\mathbb C},$ whose symbol $\beta_{k, j}$ belongs to $BMO^1_{\alpha} ({\mathbb B}_n, \mathbb C)$ and is such that $\widetilde{\beta_{k, j}}\in L^\infty (\mathbb B_n, \mathbb C),$ lies in the scalar Toeplitz algebra $\mathcal T^1_\alpha,$ i.e. for every $\epsilon >0,$ there exist positive integers $N_{k, j}$ and $M_{k, j},$ and scalar Toeplitz operators $T^{\mathbb C}_{a^{l, m}}$ with (scalar) bounded functions $a^{l, m}, 1\leq l\leq L_{k, j}, 1\leq m\leq M_{k, j},$ such that
\begin{equation}\label{t1}
\left \Vert T_{\beta_{k, j}}^{\mathbb C}-\sum_{m=1}^{M_{k, j}}\prod_{l=1}^{L_{k, j}} T^{\mathbb C}_{a^{l, m}}\right \Vert_{\mathcal L(A^2_\alpha({\mathbb B}_n, \mathbb C))} <\epsilon.
\end{equation}
In other words, for every $g\in A^2_\alpha({\mathbb B}_n, \mathbb C^d)$ such that $\left \Vert g\right \Vert_{A^2_\alpha({\mathbb B}_n, \mathbb C^d)}=1,$ 
we have:
\begin{eqnarray}\label{Td}
&&\left \Vert T_{\beta_{k, j}}^{\mathbb C}g_k-\sum_{m=1}^{M_{k, j}}\prod_{l=1}^{L_{k, j}} T^{\mathbb C}_{a^{l, m}}g_k\right \Vert_{A^2_\alpha({\mathbb B}_n, \mathbb C)}\nonumber\\
&=&\left (\int_{{\mathbb B}_n} \left \vert T_{\beta_{k, j}}^{\mathbb C}g_k(z)-\sum_{m=1}^{M_{k, j}}\prod_{l=1}^{L_{k, j}} T^{\mathbb C}_{a^{l, m}}g_k(z)\right \vert^2 {d\nu}_\alpha (z)\right )^{\frac 12}\nonumber\\
&=& \left (\int_{{\mathbb B}_n} \left \Vert T_{\beta_{k, j}}^{\mathbb C}g_k(z)e_j-\sum_{m=1}^{M_{k, j}}\prod_{l=1}^{L_{k, j}} T^{\mathbb C}_{a^{l, m}}g_k(z)e_j\right \Vert_{\mathbb C^d}^2{d\nu}_\alpha (z)\right )^{\frac 12} \nonumber\\
&=&\left (\int_{{\mathbb B}_n} \left \Vert T_{B_{k, j}}g(z)-\sum_{m=1}^{M_{k, j}}\prod_{l=1}^{L_{k, j}} T^{\mathbb C}_{a^{l, m}}g_k(z)e_j\right \Vert_{\mathbb C^d}^2{d\nu}_\alpha (z)\right )^{\frac 12}\nonumber\\
&<&\epsilon
\end{eqnarray}
by (\ref{t1}). We used (\ref{omega}) to get the last but one equality.\\
For two scalar bounded symbols $a^1$ and $a^2$ on ${\mathbb B}_n,$ let us rewrite the composition product $T^{\mathbb C}_{a^1}\circ T^{\mathbb C}_{a^2}$ of the two associated Toeplitz operators $T^{\mathbb C}_{a^1}$ and $T^{\mathbb C}_{a^2}$ in terms of the composition product of two Toeplitz operators with $\mathcal L (\mathbb C^d)-$ valued symbols. We associate with $a^2$ the symbol $A_{k, j}^2: {\mathbb B}_n\rightarrow \mathcal L(\mathbb C^d)$ defined as follows.
$$
\langle A_{k, j}^2(u) (e_l), e_m\rangle_{\mathbb C^d}=\left \{
\begin{array}{clcr}
a^2(u) &\rm if &l=k, \hskip 1truemm m=j\\
0&\rm otherwise
\end{array}
\right.
$$

For $g\in A^2_\alpha({\mathbb B}_n, \mathbb C^d)$ and $u\in {\mathbb B}_n,$ 
we have:
\begin{eqnarray*}
\langle A_{k, j}^2(u) (g(u)), e_m\rangle_{\mathbb C^d}
&=&\sum_{l=1}^d g_l(u)\langle A_{k, j}^2(u) (e_l), e_m\rangle_{\mathbb C^d}\\
&=&g_k(u)a^2(u)\delta_{j, m}\\
&=&g_k(u)a^2(u)\langle e_j, e_m\rangle_{\mathbb C^d}.
\end{eqnarray*}

This implies that $A_{k, j}^2(u) (g(u))=g_k(u)a^2(u)e_j.$
Then for $w \in {\mathbb B}_n,$ we obtain:
\begin{eqnarray}\label{akj}
T_{A_{k, j}^2} g(w)
&=&\int_{{\mathbb B}_n} A_{k, j}^2(u)(g(u))\overline{K_w^\alpha(u)}{d\nu}_\alpha (u) \nonumber\\
&=&\int_{{\mathbb B}_n} a^2(u)g_k (u)\overline{K_w^\alpha (u)}e_j{d\nu}_\alpha (u) \nonumber\\
&=&T^{\mathbb C}_{a^2}g_k(w)e_j.
\end{eqnarray}
Hence for $z\in {\mathbb B}_n,$ the following holds:
\begin{eqnarray*}
\left (T^{\mathbb C}_{a^1}\circ T^{\mathbb C}_{a^2}\right )g_k (z)e_j
&=& T^{\mathbb C}_{a^1}\left (T^{\mathbb C}_{a^2}g_k (\cdot)\right )(z)e_j\\
&=&\int_{{\mathbb B}_n} \left (\int_{{\mathbb B}_n} a^2(u)g_k (u)\overline{K_w^\alpha(u)}{d\nu}_\alpha (u)e_j\right )a^1(w)\overline{K_z^\alpha (w)}d\nu_{\alpha}(w)\\
&=&\int_{{\mathbb B}_n} T_{A_{k, j}^2} g(w)a^1(w)\overline{K_z^\alpha (w)}d\nu_{\alpha}(w),
\end{eqnarray*}
We similarly denote by $A_{j}^1: {\mathbb B}_n\rightarrow \mathcal L(\mathbb C^d)$ the symbol defined as follows.
$$
\langle A_{j}^1(w) (e_l), e_i\rangle_{\mathbb C^d}=
\left \{
\begin{array}{clcr}
a^1(w) &\rm if &l=i=j\\
0&\rm otherwise
\end{array}
\right..
$$
So for every $G\in A^2_\alpha({\mathbb B}_n, \mathbb C^d)$ and every $i\in \{1,\cdots, d\},$ we have:
\begin{eqnarray}\label{aj}
\langle A_{j}^1(w) (G(w)), e_i\rangle_{\mathbb C^d}
&=&\sum_{l=1}^d G_l (w)\langle A_{j}^1(w) (e_l), e_i\rangle_{\mathbb C^d} \nonumber\\
&=&G_j(w)\langle A_{j}^1(w) (e_j), e_i\rangle_{\mathbb C^d} \nonumber\\
&=&G_j(w)a^1(w)\delta_{i, j} \nonumber\\
&=&G_j(w)a^1(w)\langle e_j, e_i\rangle_{\mathbb C^d}.
\end{eqnarray}
This implies that
\begin{equation}\label{aj}
A_{j}^1(w) (G(w))=G_j(w)a^1(w)e_j.
\end{equation}
Next, for $g\in A^2_\alpha({\mathbb B}_n, \mathbb C^d),$ from (\ref{akj}) we obtain that
\begin{eqnarray*}
\left (T_{A_{j}^1}\circ T_{A_{k, j}^2}\right ) g(z)
&=&T_{A_{j}^1}\left (T_{A_{k, j}^2} g(\cdot) \right )(z)\\
&=&T_{A_{j}^1}\left (T^{\mathbb C}_{a^2} g_k(\cdot)e_j\right )(z).
\end{eqnarray*}

This implies that
\begin{eqnarray*}
\left (T_{A_{j}^1}\circ T_{A_{k, j}^2}\right ) g(z)
&=&\int_{{\mathbb B}_n} A_{j}^1(w) \left (T^{\mathbb C}_{a^2}g_k(w)e_j\right )\overline{K_z^\alpha (w)}d\nu_{\alpha}(w)\\
&=&\int_{{\mathbb B}_n} T^{\mathbb C}_{a^2}g_k (w)a^1(w)e_j\overline{K_z^\alpha (w)}d\nu_{\alpha}(w)\\
&=&\left (T^{\mathbb C}_{a^1}\circ T^{\mathbb C}_{a^2}\right )g_k (z)e_j.
\end{eqnarray*}
The last but one equality follows from (\ref{aj}).
We finally point out that the symbols $A_{k, j}^2$ and $A_{j}^1$ belong to $L^\infty \left ({\mathbb B}_n, \mathcal L(\mathbb C^d)\right ).$ Generalizing this process to the more general composition product $\prod_{l=1}^{L_{k, j}} T^{\mathbb C}_{a^{l, m}},$ from (\ref{Td}) we reach to the conclusion that the Toeplitz operator $T_{B_{k, j}}$ lies in the  $d-$dimensional Toeplitz algebra $\mathcal T^d_\alpha.$
\vskip 3truemm

\subsection{Proof of Theorem \ref{thm111}}\label{ssec72}
As the particular case $E=\mathbb C^d$ in Definition \ref{def17}, we recall the definition of the Berezin transform of a continuous linear operator on $A^2_\alpha (\mathbb B_n, \mathbb C^d).$

 \begin{defn}\label{def72}
Let $S\in \mathcal L(A^2_\alpha (\mathbb B_n, \mathbb C^d)).$ The Berezin transform $\widetilde S$ is the $\mathcal L(\mathbb C^d)$-valued function $$\langle \widetilde S (z)e, h\rangle_{\mathbb C^d}=\langle S\left (k_z^\alpha e\right ), k_z^\alpha h\rangle_{A^2_\alpha (\mathbb B_n, \mathbb C^d)}.$$
\end{defn}

As the particular case $E=\mathbb C^d$ in Definition \ref{def18}, we recall the definition of the transformation $S\mapsto S^z\quad (z\in \mathbb B_n)$ on $\mathcal L(A^2_\alpha (\mathbb B_n, \mathbb C^d)).$

\begin{defn}\label{def73}
For every $z\in \mathbb B_n,$ a unitary operator $U_z$ is defined on $A^2_\alpha (\mathbb B_n, \mathbb C^d)$ by 
$$\left (U_z f\right )(w)=\left (f\circ \phi_z\right )(w)k_z^\alpha (w) \quad (f\in A^2_\alpha (\mathbb B_n, \mathbb C^d), w\in \mathbb B_n).$$
For every operator $S\in \mathcal L(A^2_\alpha (\mathbb B_n, \mathbb C^d)),$ we define the operator $S^z \in \mathcal L(A^2_\alpha (\mathbb B_n, \mathbb C^d))$ by
$$S^z:=U_zSU_z.$$ 
\end{defn}


The following lemma is the particular case $E=F=\mathbb C^d$ in Lemma \ref{lem33}.

\begin{lem}\label{lem74}
Let $b \in BMO^1_{\alpha} ({\mathbb B}_n, \mathcal L(\mathbb C^d))$ be such that $\widetilde b \in L^\infty ({\mathbb B}_n, \mathcal L(\mathbb C^d)).$ For every $z\in \mathbb B_n,$ the following three identities hold.
\begin{enumerate}
\item
$\left (T_b\right )^z=T_{b\circ \phi_z};$
\item for every $e\in \mathbb C^d,$ we have
$\left \Vert \left (T_b\right )^z(e)\right \Vert_{A^2_\alpha({\mathbb B}_n, \mathbb C^d)}=\left \Vert T_b (k_z^\alpha e)\right \Vert_{A^2_\alpha({\mathbb B}_n, \mathbb C^d)}$;
\item
$\widetilde{T_b}=\widetilde b.$
\end{enumerate}
\end{lem}

We go over again the proof of Theorem \ref{thm111}. 

\begin{enumerate}
\item[$(1)\Rightarrow (2)$]
It was shown in \cite{R2014} that for every $e\in \mathbb C^d,$ the family $\{k^\alpha_z e: z\in \mathbb B_n\}$ converges weakly to zero on $A^2_\alpha({\mathbb B}_n, \mathbb C^d)$ when $|z|\rightarrow 1.$ Next, if $S$ is compact, then $\lim \limits_{|z|\rightarrow 1} \left \Vert S(k^\alpha_z e)\right \Vert_{A^2_\alpha({\mathbb B}_n, \mathbb C^d)}=0.$ The conclusion follows because $\left \Vert S(k^\alpha_z e)\right \Vert_{A^2_\alpha({\mathbb B}_n, \mathbb C^d)}=\left \Vert S^z e\right \Vert_{A^2_\alpha({\mathbb B}_n, \mathbb C^d)}.$
\item[$(2)\Rightarrow (3)$]
From Definition \ref{def72}, 
we get:
\begin{eqnarray*}
\left \Vert \widetilde S(z)\right \Vert_{\mathcal L\left (\mathbb C^d \right )}
&=&\sup \limits_{e, h\in \ell^2, \left \Vert e\right \Vert_{\mathbb C^d}=\left \Vert h\right \Vert_{\mathbb C^d}=1} \left \vert \langle \widetilde S (z)e, h\rangle_{\mathbb C^d}\right \vert\\
&\leq &\sup \limits_{e, h\in \mathbb C^d, \left \Vert e\right \Vert_{\mathbb C^d}=\left \Vert h\right \Vert_{\mathbb C^d}=1} \left \Vert S\left (k_z^\alpha e\right )\right \Vert_{A^2_\alpha (\mathbb B_n, \mathbb C^d)} \left \Vert k_z^\alpha h\right \Vert_{A^2_\alpha (\mathbb B_n, \mathbb C^d)}.
\end{eqnarray*}

The conclusion follows because $\left \Vert k_z^\alpha h\right \Vert_{A^2_\alpha (\mathbb B_n, \mathbb C^d)}=1.$
\item[$(3)\Rightarrow (1)$]
It was shown in \cite{RW2015} that $S$ is compact on $A^2_\alpha (\mathbb B_n, \mathbb C^d)$ under the assumptions $S\in \mathcal T^d_\alpha$ and $\lim \limits_{|z|\rightarrow 1} \left \Vert \widetilde S(z)\right \Vert_{\mathcal L\left (\mathbb C^d \right )}=0.$
\end{enumerate}

\subsection{Proof of Theorem 1.13}\label{ssec73}
Applying Lemma \ref{lem53}  and Corollary \ref{cor112},  we deduce the following corollary.
\begin{cor}\label{cor75}
Let $b \in BMO^1_{\alpha} ({\mathbb B}_n, \mathcal L(\ell^2))$ be such that $\widetilde b \in L^\infty ({\mathbb B}_n, \mathcal L(\ell^2)).$ Then the following three assertions are equivalent.
\begin{enumerate}
\item
$T_{b^{\llcorner (d)\lrcorner}}$ is compact on $A^2({\mathbb B}_n, \mathbb C^d);$
\item
$\widetilde{b^{\llcorner (d)\lrcorner}} (z)\rightarrow 0$ as $z\rightarrow \partial {\mathbb B}_n;$
\item For every $e\in \mathbb C^d$, we have
$\left \Vert T_{b^{\llcorner (d)\lrcorner}}\left (k_z^\alpha e \right )\right \Vert_{A^2({\mathbb B}_n, \mathbb C^d)}\rightarrow 0$ as $z\rightarrow \partial {\mathbb B}_n.$
\end{enumerate}
\end{cor}

Combining this corollary with Theorem \ref{thm57}  and Lemma \ref{lem74}, we get Theorem \ref{thm113}, which provides a four-fold sufficient condition for the compactness of $T_b.$

\section{The Toeplitz algebra $\mathcal T_{L^\infty_{fin}}.$ The proof of Proposition \ref{pro116}}\label{sec8}
Referring to section 2 of \cite{R2016}, the weighted Bergman space $A^2_\alpha (\mathbb B_n, \ell^2)$ is a strongly $\ell^2-$valued Bergman-type space on $\mathbb B_n.$  So all the results of this article are valid on this setting.

The Toeplitz algebra $\mathcal T_{L^\infty_{fin}}$ was defined in Definition \ref{def114}.

\subsection{Proof of Proposition \ref{pro116}}\label{ssec81}

By Lemma \ref{lem55}, we have the equality:
$$
M_{I^{(d)}}T_b M_{I^{(d)}}f=\mathcal J^{(d)}T_{b^{\llcorner (d)\lrcorner}}\left (M_{I^{(d)}}f\right )^{(d)}
$$
for every $f\in A^2_\alpha (\mathbb B_n,  \ell^2).$ It follows from Lemma \ref{lem53} and Theorem \ref{thm110} that the operator
$T_{b^{\llcorner (d)\lrcorner}}=\lim \limits_{m\rightarrow \infty} T_m$  in the $\mathcal L\left (A^2_\alpha (\mathbb B_n,  \mathbb C^d\right ))$-norm, where   $T_m=\sum_{l=1}^{L_m} \prod_{j=1}^{J_m} T_{A_{j, l}},$ with  $A_{j, l}\in L^\infty (\mathbb B_n, \mathcal L(\mathbb C^d)).$  First, for all positive integers $j, l$ and $f\in A^2_\alpha (\mathbb B_n, \ell^2),$ we have
\begin{equation}\label{equ81}
	\mathcal J^{(d)}T_{A_{j, l}}\left (M_{I^{(d)}}f\right )^{(d)}=T_{B_{j, l}}(f),
\end{equation}
where $B_{j, l}\in L^\infty_{fin}.$ The mapping $B_{j, l}: \mathbb B_n \rightarrow \mathcal L(\ell^2)$ is given by
$$
\langle B_{j, l}(w)(e_l), e_m\rangle_{\ell^2}=\left \{
\begin{array}{clcr}
	\langle A_{j, l}(w)(e_l), e_m\rangle_{\mathbb C^d}&\rm if&1\leq l, m\leq d\\
	0&\rm otherwise
\end{array}
\right.
.
$$
Secondly, since
\begin{eqnarray*}
	\mathcal J^{(d)}T_{b^{\llcorner (d)\lrcorner}}\left (M_{I^{(d)}}f\right )^{(d)}
	&=&\mathcal J^{(d)}\lim \limits_{J, L\rightarrow \infty} \sum_{l=1}^L \prod_{j=1}^J T_{A_{j, l}}\left (M_{I^{(d)}}f\right )^{(d)}\\
	&=&\lim \limits_{J, L\rightarrow \infty} \sum_{l=1}^L \mathcal J^{(d)}T_{A_{1, l}}T_{A_{2, l}}\cdots T_{A_{J, l}}\left (M_{I^{(d)}}f\right )^{(d)}
\end{eqnarray*}

(the limit is in the operator norm in $\mathcal L(A^2_\alpha (\mathbb B_n, \ell^2)),$ 
it suffices to show that the operator defined on $A^2_\alpha (\mathbb B_n, \ell^2)$ by
$$f\mapsto \mathcal J^{(d)}T_{A_{1, l}}T_{A_{2, l}}\cdots T_{A_{J, l}}\left (M_{I^{(d)}}f\right )^{(d)}\in \mathcal S^{(d)}$$
belongs to $\mathcal T_{L^\infty_{fin}}.$ 
	We have $\left (M_{I^{(d)}}f\right )^{(d)}\in A^2_\alpha \left (\mathbb B_n, \mathbb C^d\right )$ and hence $F:=T_{A_{2, l}}\cdots T_{A_{J, l}}\left (M_{I^{(d)}}f\right )^{(d)}\in A^2_\alpha \left (\mathbb B_n, \mathbb C^d\right ).$ 
	So $F=\left (M_{I^{(d)}}G\right )^{(d)}$ with $G=\mathcal J^{(d)} F.$  From (\ref{equ81}), we get:
	\begin{eqnarray*}
		\mathcal J^{(d)}T_{A_{1, l}}F
		&=&\mathcal J^{(d)}T_{A_{1, l}}\left (M_{I^{(d)}}G\right )^{(d)}\\
		&=&T_{B_{1, l}}(G)\\
		&=&T_{B_{1, l}}\mathcal J^{(d)} T_{A_{2, l}}\cdots T_{A_{J, l}}\left (M_{I^{(d)}}f\right )^{(d)}.
	\end{eqnarray*}
	
	It follows that :
	\begin{eqnarray*}
		\mathcal J^{(d)}T_{A_{1, l}}F
		&=&T_{B_{1, l}}\cdots T_{B_{J-1, l}}\mathcal J^{(d)}T_{A_{J, l}}\left (M_{I^{(d)}}f\right )^{(d)}\\
		&=&T_{B_{1, l}}T_{B_{2, l}}\cdots T_{B_{J, l}}f.
	\end{eqnarray*}

	This completes the proof of Proposition \ref{pro116}.

\subsection{A corollary}\label{ssec82}
From Proposition \ref{pro116}, Theorem \ref{thm115}, Proposition \ref{thm45} and Theorem \ref{thm110}, we easily deduce the following corollary.
\begin{cor}\label{cor81}
Let $b\in BMO^1_\alpha ({\mathbb B}_n, \mathcal L(\ell^2))$ be such that $\widetilde b\in L^\infty ({\mathbb B}_n, \mathcal L(\ell^2)).$ Suppose further that 
$$\lim \limits_{d\rightarrow \infty} \left \Vert \widetilde{\left \Vert b_{(d)}\right \Vert_{\mathcal L(\ell^2)}}\right \Vert_{L^\infty ({\mathbb B}_n, \mathbb C)}=
\lim \limits_{d\rightarrow \infty} \left \Vert \widetilde{\left \Vert {\left ((b(\cdot))^\star_{\mathcal L(\ell^2)}\right )_{(d)}}\right \Vert_{\mathcal L(\ell^2)}}\right \Vert_{L^\infty ({\mathbb B}_n, \mathbb C)}=0.$$
If 
\begin{equation}\label{iv}
\sup \limits_{e\in \mathbb C^d, \left \Vert e\right \Vert_{\mathbb C^d}=1} \lim \limits_{z\rightarrow \partial {\mathbb B}_n} \left \Vert T_b\left (k^\alpha_ze \right )\right \Vert_{A^2_\alpha({\mathbb B}_n, \ell^2)}=0
\end{equation}
for every positive integer $d,$ then $T_b$ is compact.
\end{cor}

\begin{rem}\label{rem82}
Under the hypotheses (ii) and (iii) of Theorem \ref{thm113}, the assertion (iv) of this theorem and  the last condition (\ref{iv})
of Corollary \ref{cor81} amount to the same, since $T_{b^{\llcorner (d)\lrcorner}}\simeq M_{I^{(d)}}T_bM_{I^{(d)}}$ according to Proposition \ref{pro56}, and 
$$\lim \limits_{d\rightarrow \infty} \left \Vert T_b-M_{I^{(d)}}T_bM_{I^{(d)}}\right \Vert_{\mathcal L\left (A^2_\alpha({\mathbb B}_n, \ell^2)\right )}=0.$$
\end{rem}

\section{Theorem \ref{thm71} (Corollary 3.6 of Sadeghi-Zorboska) without extra conditions does not extend to the infinite dimension}\label{sec9}
In this section, we prove the following theorem.
\begin{thm}\label{thm91}
The hypotheses $b\in BMO^1_{\alpha} ({\mathbb B}_n, \mathcal L(\ell^2))$ and $\widetilde {b} \in L^\infty ({\mathbb B}_n, \mathcal L(\ell^2))$ do not imply that $T_b\in \mathcal T_{L^\infty_{fin}}.$
\end{thm}

In other words, without extra conditions, Theorem \ref{thm71} (\cite[Corollary 3.6]{SZ2020}) for dimension one and Theorem \ref{thm110} in finite dimension do not extend to the infinite dimension.

\begin{proof}
Our counterexample is given by a symbol of the form $b=\tau I,$ where $\tau \in BMO^1_{\alpha} ({\mathbb B}_n, \mathbb C)$ is such that $\widetilde \tau \in L^\infty ({\mathbb B}_n, \mathbb C).$ We shall need the following lemma.

\begin{lem}\label{lem92}
Let $\tau \in L^1_\alpha \left (\mathbb B_n, \mathbb C\right ).$ For every $z\in {\mathbb B}_n$ and every $e\in \ell^2$ such that $\left \Vert e\right \Vert_{\ell^2}=1,$ the following identities hold.
\begin{enumerate}
\item
$T_{\tau I}(k_z^\alpha e)(\zeta)=\left (T^{\mathbb C}_{\tau}k_z^\alpha (\zeta) \right )e \quad (\zeta \in {\mathbb B}_n);$
\item
$\left \Vert T_{\tau I}(k_z^\alpha e)\right \Vert_{A^2_\alpha (\mathbb B_n,  \ell^2)}=\left \Vert T^{\mathbb C}_{\tau}k_z^\alpha\right \Vert_{A^2_\alpha ({\mathbb B}_n, \mathbb C)},$
\end{enumerate}
where $T_\tau^{\mathbb C}$ denotes the scalar Toeplitz operator with scalar symbol $\tau.$
\end{lem}

\begin{proof}

\begin{enumerate}
\item
	The proof uses (\ref{eq:26}) with $f(w)=k^\alpha_z(w)$.

\item
It follows from (1) that
\begin{eqnarray*}
\left \Vert T_{\tau I}(k_z^\alpha e)\right \Vert_{A^2_\alpha (\mathbb B_n,  \ell^2)}^2
&=&\int_{{\mathbb B}_n} \left \vert T^{\mathbb C}_{\tau}k_z^\alpha\right \vert^2{d\nu}_\alpha\\
&=&\left \Vert T^{\mathbb C}_{\tau}k_z^\alpha\right \Vert_{A^2_\alpha({\mathbb B}_n, \mathbb C)}^2.
\end{eqnarray*}

\end{enumerate}
\end{proof}

We next prove the theorem. We take $\tau=\chi_{r\mathbb B_n}$ with $r\in (0, 1).$ For this function, we have $\widetilde \tau (z)\rightarrow 0 \quad (z\rightarrow \partial \mathbb B_n)$ or equivalently in view of \cite{ZLL2011}, $\lim \limits_{z\rightarrow \partial {\mathbb B}_n} \left \Vert T^{\mathbb C}_{\tau}k_z^\alpha\right \Vert_{A^2_\alpha({\mathbb B}_n, \mathbb C)}=0,$ which by assertion (2) of Lemma \ref{lem92}, is equivalent to
$\lim \limits_{z\rightarrow \partial {\mathbb B}_n} \left \Vert T_{\tau I}(k_z^\alpha e)\right \Vert_{A^2_\alpha (\mathbb B_n,  \ell^2)}^2=0.$ Suppose on the contrary that $T_b\in \mathcal T_{L^\infty_{fin}}.$ Then by Theorem \ref{thm115}, $T_b,$ with $b=\tau I,$ must be compact on $A^2_\alpha (\mathbb B_n,  \ell^2).$ We have reached to a contradiction because according to Proposition \ref{pro24} and Corollary \ref{cor25}, this is not the case unless $\tau$ vanishes identically on $\mathbb B_n.$ 
\end{proof}
	\section{Sufficiently localized operators}\label{sec10}
\subsection{The Space $BMO^1_{\alpha} ({\mathbb B}_n, \mathcal L(\ell^2, \ell^1\cap \ell^2))$}\label{ssec101}
Two Banach spaces $E$ and $F$ are said to be compatible if they are continuously embedded in a Hausdorff topological vector space $Z.$ In this case, equipped with the norm $\left \Vert x\right \Vert :=\max \left \{||x||_E, ||x||_F\right \},$
the vector space $E\cap F$ is a Banach space. An example is given by the Banach sequence spaces $\ell^1$ and $\ell^2$ which are continuously embedded  in the  Hausdorff topological vector space consisting of infinite complex sequences, endowed with the convergence in measure (the counting measure). Equipped with the norm $\left \Vert x\right \Vert :=\max \left \{||x||_{\ell^1}, ||x||_{\ell^2}\right \},$ the space $\ell^1\cap \ell^2$ is then a Banach space.

\begin{lem}\label{lem101}
The following identity holds.
$$BMO^1_{\alpha} \left ({\mathbb B}_n, \mathcal L(\ell^2, \ell^1\cap \ell^2)\right )= BMO^1_{\alpha} \left ({\mathbb B}_n, \mathcal L(\ell^2)\right )\cap BMO^1_{\alpha}\left ({\mathbb B}_n, \mathcal L(\ell^2, \ell^1)\right )$$
with equivalence of norms.
\end{lem} 

\begin{proof}
\begin{enumerate}
\item
We first show that $\mathcal L(\ell^2, \ell^1\cap \ell^2)= \mathcal L(\ell^2)\cap \mathcal L(\ell^2, \ell^1)$ with equivalence of norms. For $\Phi \in \mathcal L(\ell^2, \ell^1\cap \ell^2),$ we have
\begin{eqnarray*}
\left \Vert \Phi\right \Vert_{\mathcal L(\ell^2, \ell^1\cap \ell^2)}
&=&\sup \limits_{\left \Vert e\right \Vert_{\ell^2}=1} \left \Vert \Phi(e)\right \Vert_{\ell^1\cap \ell^2}\\
&=&\sup \limits_{\left \Vert e\right \Vert_{\ell^2}=1} \max \left \{\left \Vert \Phi(e)\right \Vert_{\ell^1}, \left \Vert \Phi(e)\right \Vert_{\ell^2} \right \}\\
&\geq& \sup \limits_{\left \Vert e\right \Vert_{\ell^2}=1} \left \Vert \Phi(e)\right \Vert_{\ell^r}\qquad (r=1, 2)\\
&=&\left \Vert \Phi\right \Vert_{\mathcal L(\ell^2, \ell^r)} \qquad (r=1, 2). 
\end{eqnarray*}

So
\begin{eqnarray*}
\left \Vert \Phi\right \Vert_{\mathcal L(\ell^2, \ell^1\cap \ell^2)}
&\geq& \max \left \{\left \Vert \Phi\right \Vert_{\mathcal L(\ell^2)},\left \Vert \Phi\right \Vert_{\mathcal L(\ell^2, \ell^1)} \right \}\\
&=&\left \Vert \Phi\right \Vert_{\mathcal L(\ell^2)\cap  \mathcal L(\ell^2, \ell^1)}. 
\end{eqnarray*}

Conversely, for $\Phi \in \mathcal L(\ell^2)\cap \mathcal L(\ell^2, \ell^1),$ we have
\begin{eqnarray*}
\left \Vert \Phi\right \Vert_{\mathcal L(\ell^2)\cap  \mathcal L(\ell^2, \ell^1)}
&=&\max \left \{\left \Vert \Phi\right \Vert_{\mathcal L(\ell^2)},\left \Vert \Phi\right \Vert_{\mathcal L(\ell^2, \ell^1)} \right \}\\
&=&\max \left \{\sup \limits_{\left \Vert e\right \Vert_{\ell^2}=1} \left \Vert \Phi(e)\right \Vert_{\ell^2}, \sup \limits_{\left \Vert e\right \Vert_{\ell^2}=1} \left \Vert \Phi(e)\right \Vert_{\ell^1}    \right \}\\
&\geq& \frac 12\left \{\sup \limits_{\left \Vert e\right \Vert_{\ell^2}=1} \left \Vert \Phi(e)\right \Vert_{\ell^2}+ \sup \limits_{\left \Vert e\right \Vert_{\ell^2}=1} \left \Vert \Phi(e)\right \Vert_{\ell^1} \right \}\\
&\geq & \frac 12 \sup \limits_{\left \Vert e\right \Vert_{\ell^2}=1} \left \{\left \Vert \Phi(e)\right \Vert_{\ell^2}+\left \Vert \Phi(e)\right \Vert_{\ell^1} \right \}\\
&\geq & \frac 12 \sup \limits_{\left \Vert e\right \Vert_{\ell^2}=1} \max \left \{\left \Vert \Phi(e)\right \Vert_{\ell^2},\left \Vert \Phi(e)\right \Vert_{\ell^1} \right \}\\
&=&\frac 12 \sup \limits_{\Vert e\Vert_{\ell^2}=1} \left \Vert \Phi(e)\right \Vert_{\ell^1\cap \ell^2}= \frac 12 \left \Vert \Phi\right \Vert_{\mathcal L(\ell^2, \ell^1\cap \ell^2)}.
\end{eqnarray*}

\item
We next show that
 $$BMO^1_{\alpha} \left ({\mathbb B}_n, \mathcal L(\ell^2, \ell^1\cap \ell^2)\right )= BMO^1_{\alpha} \left ({\mathbb B}_n, \mathcal L(\ell^2)\right )\cap BMO^1_{\alpha}\left ({\mathbb B}_n, \mathcal L(\ell^2, \ell^1)\right )$$
with equivalence of norms. For $b\in BMO^1_{\alpha} \left (\mathbb B_n, \mathcal L(\ell^2, \ell^1\cap \ell^2)\right ),$ in view of the definition and the first part of the proof, we have, for every $r=1, 2.$ :
\begin{eqnarray*}
\left \Vert b\right \Vert_{BMO^1_{\alpha} \left({\mathbb B}_n, \mathcal L(\ell^2, \ell^1\cap \ell^2\right ))}
&=&\sup_{z\in \mathbb B_n} \int_{{\mathbb B}_n} \left \Vert b\circ \varphi_z (w)-\tilde b(z)\right \Vert_{\mathcal L(\ell^2, \ell^1\cap \ell^2)}d\nu_{\alpha}(w) \\
&\geq & \sup_{z\in \mathbb B_n} \int_{\mathbb B_n} \left\Vert b\circ \varphi_z (w)-\tilde{b}(z)\right\Vert_{\mathcal L(\ell^2, \ell^r)}d\nu_{\alpha}(w) \\
&=&\left \Vert b\right \Vert_{BMO^1_{\alpha} \left(\mathbb B_n, \mathcal L(\ell^2, \ell^r)\right)}
\end{eqnarray*}
We deduce that
$\left \Vert b\right \Vert_{BMO^1_{\alpha} \left ({\mathbb B}_n, \mathcal L(\ell^2, \ell^1\cap \ell^2\right ))}\geq \left \Vert b\right \Vert_{BMO^1_{\alpha} \left ({\mathbb B}_n, \mathcal L(\ell^2)\right )\cap BMO^1_{\alpha}\left ({\mathbb B}_n, \mathcal L(\ell^2, \ell^1)\right )}.$

Conversely, for $b\in BMO^1_{\alpha} \left ({\mathbb B}_n, \mathcal L(\ell^2)\right )\cap BMO^1_{\alpha}\left ({\mathbb B}_n, \mathcal L(\ell^2, \ell^1)\right ),$ 
we have:
\begin{eqnarray*}
&&\left \Vert b\right \Vert_{BMO^1_{\alpha} \left ({\mathbb B}_n, \mathcal L(\ell^2)\right )\cap BMO^1_{\alpha}\left ({\mathbb B}_n, \mathcal L(\ell^2, \ell^1)\right )} \\
&=& \max \left \lbrace\Vert b \Vert_{BMO^1_{\alpha} \left ({\mathbb B}_n, \mathcal L(\ell^2)\right )}, \left \Vert b\right \Vert_{BMO^1_{\alpha} \left({\mathbb B}_n, \mathcal L(\ell^2, \ell^1) \right)}\right\rbrace \\
&=&\max \left \{\sup \limits_{z\in \mathbb B_n} \int_{\mathbb B_n} \left \Vert b\circ \varphi_z (w)-\tilde b(z)\right \Vert_{\mathcal L(\ell^2)}d\nu_{\alpha}(w), \sup \limits_{z\in \mathbb B_n} \int_{\mathbb B_n} \left \Vert b\circ \varphi_z (w)-\tilde b(z)\right \Vert_{\mathcal L(\ell^2, \ell^1)}d\nu_{\alpha}(w)\right \} \\
&\geq & \frac 12\left \{\sup \limits_{z\in {\mathbb B}_n} \int_{{\mathbb B}_n} \left \Vert b\circ \varphi_z (w)-\tilde b(z)\right \Vert_{\mathcal L(\ell^2)}d\nu_\alpha (z)+ \sup \limits_{z\in {\mathbb B}_n} \int_{{\mathbb B}_n} \left \Vert b\circ \varphi_z (w)-\tilde b(z)\right \Vert_{\mathcal L(\ell^2, \ell^1)}d\nu_{\alpha}(w)\right \} \\
&\geq & \frac 12\left \{\sup \limits_{z\in {\mathbb B}_n} \left \{\int_{{\mathbb B}_n} \left (\left \Vert b\circ \varphi_z (w)-\tilde b(z)\right \Vert_{\mathcal L(\ell^2)}+ \left \Vert b\circ \varphi_z (w)-\tilde b(z)\right \Vert_{\mathcal L(\ell^2, \ell^1)}\right )d\nu_{\alpha}(w)\right \}\right \} \\
&\geq & \frac 12 \sup \limits_{z\in {\mathbb B}_n} \int_{{\mathbb B}_n} \max \left \{\left \Vert b\circ \varphi_z (w)-\tilde b(z)\right \Vert_{\mathcal L(\ell^2)}, \left \Vert b\circ \varphi_z (w)-\tilde b(z)\right \Vert_{\mathcal L(\ell^2, \ell^1)} \right \}d\nu_{\alpha}(w)\\
&=& \frac 12\sup \limits_{z\in {\mathbb B}_n} \int_{{\mathbb B}_n} \left \Vert b\circ \varphi_z (w)-\tilde b(z)\right \Vert_{\mathcal L(\ell^2)\cap \mathcal L(\ell^2, \ell^1)}d\nu_{\alpha}(w) \\
&\geq & \frac 14\sup \limits_{z\in {\mathbb B}_n} \int_{{\mathbb B}_n} \left \Vert b\circ \varphi_z (w)-\tilde b(z)\right \Vert_{\mathcal L(\ell^2, \ell^1\cap \ell^2)}d\nu_{\alpha}(w)\\
&=&\frac 14\left \Vert b\right \Vert_{BMO^1_{\alpha} \left ({\mathbb B}_n, \mathcal L(\ell^2, \ell^1\cap \ell^2\right ))}.
\end{eqnarray*}

For the last inequality, we used the first part of the proof.
\end{enumerate}
\end{proof}		

\subsection{Sufficiently localized operators. Proof of Theorem \ref{thm119}}\label{ssec102}
Sufficiently localized operators were defined in Definition \ref{def117}. Here, we consider the case where $E=F=\ell^2.$
The left-hand side of (\ref{3.2}) can be written as follows:

\begin{eqnarray*}
\left (\int_{{\mathbb B}_n} \left (\sum_{i=1}^\infty \left \vert \langle S^z e_j(u), e_i\rangle_{\ell^2}\right \vert \right )^p{d\nu}_\alpha (u)   \right )^{\frac 1p}
&=&\left (\int_{{\mathbb B}_n} \left \Vert S^z e_j)(u)\right \Vert_{\ell^1}^p{d\nu}_\alpha (u)   \right )^{\frac 1p}\\
&=&\left \Vert S^z e_j)\right \Vert_{L^p_\alpha ({\mathbb B}_n, \ell^1)}.
\end{eqnarray*}
		
In particular, for $b\in L^1_\alpha (\mathbb B_n, \mathcal L(\ell^2)),$ according to assertion (1) of Lemma \ref{lem33}, the Toeplitz operator $T_b$ satisfies the condition (\ref{3.2}) if
$$\sup \limits_{j} \hskip 1truemm \sup \limits_{z\in {\mathbb B}_n} \left \Vert T_{b\circ \varphi_z}(e_j)\right \Vert_{A^p_\alpha ({\mathbb B}_n, \ell^1)}<\infty$$
for some real number $p\textgreater \frac {n+2+2\alpha}{1+\alpha}.$

We next prove Theorem \ref{thm119}. We shall show that under the hypotheses, $T_b$ satisfies the conditions (\ref{3.2}) and (\ref{3.1}), in Definition \ref{def117}, for every positive number $p.$

Let $p$ be an arbitrary positive number. For condition (\ref{3.2}), we prove more precisely  that for every positive integer $j$ and every $z\in {\mathbb B}_n,$ the following estimate holds. 
$$\left \Vert T_{b\circ \varphi_z} e_j\right \Vert_{L^p_\alpha ({\mathbb B}_n, \ell^1)}\lesssim ||| b|||_{BMO^1_{\alpha} ({\mathbb B}_n, \mathcal L(\ell^2, \ell^1))},$$
where $||| b|||_{BMO^1_{\alpha} ({\mathbb B}_n, \mathcal L(\ell^2, \ell^1))}=\Vert \widetilde{b}(0)\Vert_{\mathcal L(\ell^2, \ell^1)} + \Vert b\Vert_{BMO^1_{\alpha} ({\mathbb B}_n, \mathcal L(\ell^2, \ell^1))}. $
We shall need the following lemma.
\begin{lem}\label{lem102}
\begin{enumerate}
\item[(1)]
For every $b \in L^1_\alpha ({\mathbb B}_n, \mathcal L(\ell^2, \ell^1))$ and all $z, w\in {\mathbb B}_n,$ we have the equality:
$$\widetilde{b\circ \varphi_z} (w)=\widetilde b\circ \varphi_z (w).$$
\item[(2)]
Let $b \in BMO^1_{\alpha} ({\mathbb B}_n, \mathcal L(\ell^2, \ell^1)).$ For all $z\in {\mathbb B}_n$ and $j=1, 2, \cdots,$ we have $(b\circ \varphi_z)(e_j)\in BMO^1_{\alpha} ({\mathbb B}_n, \ell^1).$ More precisely, the following estimate holds.
$$\left \Vert (b\circ \varphi_z)(e_j)\right \Vert_{BMO^1_{\alpha} ({\mathbb B}_n, \ell^1)}\leq \left \Vert b\right \Vert_{BMO^1_{\alpha} ({\mathbb B}_n, \mathcal L(\ell^2, \ell^1))}.$$
\end{enumerate}
\end{lem}

\begin{proof}[Proof of the lemma]
\begin{enumerate}
\item[(1)]
The point is the identity $\left \vert k_w\left (\varphi_z(\zeta)\right )\right \vert \left \vert k_z(\zeta)\right \vert=\left \vert k_{\varphi_z (w)} (\zeta)\right \vert$ \cite{ZLL2011} (cf. also \cite[p. 122]{D2024}) and the change of variables formula \cite[Proposition 1.13]{Z2005}:
$$\int_{\mathbb B_n} b\circ \varphi_a (z)d\nu_\alpha (z)=\int_{\mathbb B_n} b(z)\frac {(1-|a|^2)^{n+1+\alpha}}{\left \vert 1-\langle z, a\rangle\right \vert^{2(n+1+\alpha)}}d\nu_\alpha (z)$$
\item[(2)]
We easily check that $\varphi_z \circ \varphi_w=\varphi_{\varphi_z (w)}\circ U,$ where $U$ is a unitary transformation of $\mathbb C^n.$ Combining with (1), we obtain:

\begin{eqnarray*}
&&\left \Vert  \left (\left (b\circ \varphi_z\right )\circ \varphi_w-\widetilde {b\circ \varphi_z} (w)\right )(e_j)\right \Vert_{L^1_\alpha ({\mathbb B}_n, \ell^1)}\\
&=&\int_{{\mathbb B}_n} \left \Vert \left (b\circ \varphi_{\varphi_z (w)}(u)-\widetilde b\circ \varphi_z (w) \right )(e_j)\right \Vert_{\ell^1}{d\nu}_\alpha (u)\\
&\leq & \int_{{\mathbb B}_n} \left \Vert b\circ \varphi_{\varphi_z (w)}(u)-\widetilde b\circ \varphi_z (w)\right \Vert_{\mathcal L(\ell^2, \ell^1)}{d\nu}_\alpha (u)\\
&=&\left \Vert b\circ \varphi_{\varphi_z (w)}-\widetilde b\circ \varphi_z (w)\right \Vert_{L^1_\alpha \left ({\mathbb B}_n, \mathcal L(\ell^2, \ell^1)\right )}\\
&\leq & \left \Vert b\right \Vert_{BMO^1_{\alpha} ({\mathbb B}_n, \mathcal L(\ell^2, \ell^1))}<\infty.
\end{eqnarray*}
Taking the supremum with respect to $w\in {\mathbb B}_n$ gives the announced estimate.
\end{enumerate}
\end{proof}

We recall the following definition.

\begin{defn}\label{def103}
A holomorphic function $f: \mathbb B_n\rightarrow \ell^1$ belongs to the Bloch space $\mathfrak B({\mathbb B}_n, \ell^1)$ if
$$\left \Vert f\right \Vert_{\mathfrak B({\mathbb B}_n, \ell^1)}:=\sup \limits_{z\in \mathbb B_n} \left (1-|z|^2\right )\left \Vert \nabla f(z)\right \Vert_{\left (\ell^1\right )^n} \textless \infty,$$
where $\nabla$ is the complex gradient in $\mathbb C^n,$ i.e.
$$\nabla f=\left (\frac {\partial}{\partial z_1}f,\cdots, \frac {\partial}{\partial z_n}f \right ),$$
and $\left (\ell^1\right )^n=\ell^1\times \cdots \times \ell^1 \hskip 2 truemm (n \hskip 2truemm \rm{times}).$
\end{defn}




We recall the following vector-valued generalizations  due to Defo \cite{DT2024, D2024}. The second one is a generalization of a theorem of Li and Luecking \cite{LL1994}.
\begin{thm}\label{thm104}
	The following assertions hold.
\begin{enumerate}
\item[(1)]
The Bloch space $\mathfrak B({\mathbb B}_n, \ell^1)$ embeds continuously in $L^p_\alpha ({\mathbb B}_n, \ell^1)$ for every positive exponent $p,$ in the sense that, for every $g\in \mathfrak B({\mathbb B}_n, \ell^1),$ the following estimate holds
$$\left \Vert g\right\Vert_{A^p_\alpha(\mathbb B_n, \ell^1)}\lesssim \left \Vert g\right\Vert_{\mathfrak B ({\mathbb B}_n, \ell^1)}.$$
\item[(2)]
The Bergman projector $P_\alpha$ extends to a continuous operator from $BMO^1_{\alpha} ({\mathbb B}_n, \ell^1)$ to the Bloch space $\mathfrak B({\mathbb B}_n, \ell^1),$ in the sense that, for every $f\in  BMO^1_\alpha ({\mathbb B}_n, \ell^1),$ the following estimate holds
$$\left \Vert P_\alpha f\right\Vert_{\mathfrak B({\mathbb B}_n, \ell^1)}\lesssim \left \Vert f\right\Vert_{ BMO^1_\alpha ({\mathbb B}_n, \ell^1)}.$$
\end{enumerate}
\end{thm}
From assertion (1) of Theorem \ref{thm104}, we get:
\begin{eqnarray*}
\left \Vert T_{b\circ \varphi_z}(e_k)\right \Vert_{A^p_\alpha (\mathbb B_n, \ell^1)}
&=&\left \Vert P_\alpha \left (\left (b\circ \varphi_z\right ) (e_k)\right )\right \Vert_{A^p_\alpha ({\mathbb B}_n, \ell^1)}\\
&\lesssim & \left \Vert P_\alpha \left (\left (b\circ \varphi_z\right )(e_k)\right ) \right \Vert_{\mathfrak B ({\mathbb B}_n, \ell^1)}\\
&\lesssim &\left \Vert \left (b\circ \varphi_z\right )(e_k)\right \Vert_{BMO^1_{\alpha} ({\mathbb B}_n, \ell^1)}
\end{eqnarray*}

The latter inequality follows from assertion (2) of Theorem \ref{thm104}. By assertion (2) of Lemma \ref{lem102}, we conclude that
$$\left \Vert T_{b\circ \varphi_z} (e_k)\right \Vert_{L^p_\alpha ({\mathbb B}_n, \ell^1)}\lesssim \left \Vert b\right \Vert_{BMO^1_{\alpha} ({\mathbb B}_n, \mathcal L(\ell^2, \ell^1))}.$$

To prove that $T_b$ satisfies condition (\ref{3.1}) for every positive number $p,$ we recall that according to Lemma \ref{lem47}, if $b \in BMO^1_{\alpha} (\mathbb B_n, \mathcal L(\ell^2))$ is such that $\widetilde b \in L^\infty (\mathbb B_n, \mathcal L(\ell^2)),$ then $T_b^\star=T_{\left (b(\cdot)\right )^\star_{\mathcal L(\ell^2)}}.$ We then apply the first part of the proof replacing $T_b$ by $T_b^\star=T_{\left (b(\cdot)\right)^\star_{\mathcal L(\ell^2)}}.$

\section{The classes of examples of Toeplitz operators from Subsection \ref{ssec22}  and Section \ref{sec5} revisited}\label{sec11}

\subsection{Working again the class of examples from Subsection \ref{ssec22}}\label{ssec111}
The symbol $b=\tau I$ studied in Section \ref{sec2} and Section \ref{sec9} does not belong to $BMO^1_\alpha ({\mathbb B}_n, \mathcal L(\ell^2, \ell^1))$ unless $\tau$ is constant a.e., and its associated Toeplitz operator $T_b$ does not satisfy the estimate $\limsup \limits_{d\rightarrow \infty} \left \Vert T_b M_{I_{(d)}}\right \Vert_{\mathcal L(A^2_\alpha (\mathbb B_n,  \ell^2))}=0$ unless the symbol $\tau$ vanishes identically. Indeed, on the one hand, for all $z, w\in {\mathbb B}_n,$ we have:

\begin{eqnarray*}
\left \Vert b\circ \varphi_z (w)-\widetilde b(z)\right \Vert_{\mathcal L(\ell^2, \ell^1))}&=&\sup \limits_{\left \Vert e\right \Vert_{\ell^2}=1} \left \Vert \left (b\circ \varphi_z(w)-\widetilde b(z)\right )(e)\right \Vert_{\ell^1}\\
&=&\left \vert \tau \circ \varphi_z-\widetilde \tau(z)\right \vert \cdot \sup \limits_{\left \Vert e\right \Vert_{\ell^2}=1} \left \Vert e\right \Vert_{\ell^1}.
\end{eqnarray*}
Taking $e$ such that for every positive integer $i,$ we have $\langle e, e_i\rangle_{\ell^2}=\frac {\sqrt 6}{\pi}\frac 1i,$ the following holds: $e\in \ell^2$ because $\left \Vert e\right \Vert_{\ell^2}^2=\frac 6{\pi^2}\sum_{i=1}^\infty \frac 1{i^2}=1,$ while $\left \Vert e\right \Vert_{\ell^1}=\frac {\sqrt 6}{\pi}\sum_{i=1}^\infty \frac 1i=\infty.$ So $ \sup \limits_{\left \Vert e\right \Vert_{\ell^2}=1} \left \Vert e\right \Vert_{\ell^1}=\infty.$ The only way forward is that for every $z\in {\mathbb B}_n,$ we have $\tau \circ \varphi_z(w)=\widetilde \tau(z)$ for almost every $w\in \mathbb B_n.$ In particular, for $z=0,$ we have $\tau (w)=\widetilde \tau (0)$ for almost every $w\in {\mathbb B}_n.$ In other words, $\tau$ is constant a.e.\\
On the other hand, for every scalar function $f\in H^\infty (\mathbb B_n, \mathbb C)$ such that $\left \Vert f\right \Vert_{A^2_\alpha (\mathbb B_n, \mathbb C)}=1,$ the  $\ell^2$-valued function $f(z)e_{d+1}$ belongs to $A^2_\alpha (\mathbb B_n,  \ell^2)$ with norm 1. We have $M_{I_{(d)}}\left (fe_{d+1}\right )=fe_{d+1}$ and so
$$T_b M_{I_{(d)}} \left (fe_{d+1}\right )(z)=\int_{{\mathbb B}_n} \tau (w)f(w)e_{d+1}\overline{K_z^\alpha (w)}d\nu_\alpha (w)=\left (T^{\mathbb C}_\tau f \right ) (z)e_{d+1}.$$
Then $\left \Vert T_b M_{I_{(d)}} f\right \Vert_{A^2_\alpha (\mathbb B_n,  \ell^2)}=\left \Vert T^{\mathbb C}_\tau f \right \Vert_{A^2_\alpha({\mathbb B}_n, \mathbb C)}.$ 
This implies that 
$$\left \Vert T_b M_{I_{(d)}}\right \Vert_{\mathcal L(A^2_\alpha (\mathbb B_n,  \ell^2))}\geq \left \Vert T^{\mathbb C}_\tau f\right \Vert_{A^2_\alpha({\mathbb B}_n, \mathbb C)}.$$
 Now, assume that $\limsup \limits_{d\rightarrow \infty} \left \Vert T_b M_{I_{(d)}}\right \Vert_{\mathcal L(A^2_\alpha (\mathbb B_n,  \ell^2))}=0.$ Then $\left \Vert T^{\mathbb C}_\tau f \right \Vert_{A^2_\alpha({\mathbb B}_n, \mathbb C)}=0$ for every scalar function $f\in H^\infty (\mathbb B_n, \mathbb C).$ So the (scalar) bounded Toeplitz operator $T^{\mathbb C}_\tau$ is the zero operator and hence, as shown in the end of the proof of Proposition \ref{pro24},  $\tau \equiv 0$ a.e. on $\mathbb B_n.$  

\subsection{Working again the class of examples from Section \ref{sec5}}\label{ssec112}
We show that the symbol $b$ given in  Section \ref{sec6} also fulfills the hypotheses of Corollary \ref{cor120}. The hypotheses $b \in BMO^1_{\alpha} ({\mathbb B}_n, \mathcal L(\ell^2)), \widetilde b \in L^\infty ({\mathbb B}_n, \mathcal L(\ell^2))$ and $\lim \limits_{d\rightarrow \infty} \left \Vert \widetilde{\left \Vert b_{(d)}(\cdot)\right \Vert_{\mathcal L(\ell^2)}}\right \Vert_{L^\infty ({\mathbb B}_n)}=0
$ were already proved in Theorem \ref{thm61}. It remains to show the following three assertions:
\begin{enumerate}
\item
$b \in BMO^1_{\alpha} ({\mathbb B}_n, \mathcal L(\ell^2, \ell^1));$
\item
$(b(\cdot))^\star_{\mathcal L(\ell^2)}\in BMO^1_{\alpha} ({\mathbb B}_n, \mathcal L(\ell^2, \ell^1));$
\item
$\sup \limits_{e\in \mathbb C^d, \left \Vert e\right \Vert_{\mathbb C^d}=1} \lim \limits_{z\rightarrow \partial {\mathbb B}_n} \left \Vert T_b\left (k_z^\alpha e \right )\right \Vert_{A^2_\alpha (\mathbb B_n,  \ell^2)}=0$ for every positive integer $d.$
\end{enumerate} 
\begin{enumerate}
\item
We first show that $b \in BMO^1_\alpha ({\mathbb B}_n, \mathcal L(\ell^2, \ell^1)).$ We have

\begin{eqnarray*}
\left \Vert b\circ \varphi_z (w)-\widetilde b(z)\right \Vert_{\mathcal L(\ell^2, \ell^1)}
&=&\sup \limits_{\left \Vert e\right \Vert_{\ell^2}=1} \left \Vert \left (b\circ \varphi_z (w)-\widetilde b(z)\right )(e)\right \Vert_{\ell^1}\\
&=&\sup \limits_{\sum_{i=1}^\infty \left \vert \lambda_i\right \vert^2=1}\sum_{j=1}^\infty  \left \vert \sum_{i=1}^\infty \lambda_i \langle  \left (b \circ \varphi_z (w)-\widetilde b (z) \right )(e_i), e_j\rangle_{\ell^2}\right \vert\\
&=&\sup \limits_{\sum_{i=1}^\infty \left \vert \lambda_i\right \vert^2=1}\sum_{j=1}^\infty  \left \vert \sum_{i=1}^\infty \lambda_i \frac {\tau \circ \varphi_z (w)-\widetilde \tau (z)}{2^i}\langle  e_i, e_j\rangle_{\ell^2}\right \vert\\
&=&\sup \limits_{\sum_{i=1}^\infty \left \vert \lambda_i\right \vert^2=1}\sum_{j=1}^\infty  \left \vert  \lambda_j \frac {\tau \circ \varphi_z (w)-\widetilde \tau (z)}{2^j}\right \vert\\
&=& \left \vert \tau \circ \varphi_z (w)-\widetilde \tau (z)\right \vert\sup \limits_{\sum_{i=1}^\infty \left \vert \lambda_i\right \vert^2=1}   \sum_{j=1}^\infty\frac {\left \vert \lambda_j \right \vert}{2^j}\\
&\leq& \left \vert \tau \circ \varphi_z (w)-\widetilde \tau (z)\right \vert.
\end{eqnarray*}

For the latter inequality, we used the inequality $\left \vert \lambda_j \right \vert \leq 1.$
The hypothesis $\tau \in BMO^1_{\alpha} ({\mathbb B}_n, \mathbb C)$ then implies the announced conclusion.
\item
We next show that $(b(\cdot))^\star_{\mathcal L(\ell^2)}\in BMO^1_{\alpha} ({\mathbb B}_n, \mathcal L(\ell^2, \ell^1)).$ We first determine $(b(\cdot))^\star_{\mathcal L(\ell^2)}.$ For all $e, h\in \ell^2$ and for every positive integer $i,$ we have:
\begin{eqnarray*}
\langle e_i, (b(w))^\star_{\mathcal L(\ell^2)} (h)\rangle_{\ell^2}
&=&\langle b(w) (e_i), h\rangle_{\ell^2}\\
&=&\frac {\tau (w)}{2^i} \langle e_i, h\rangle_{\ell^2}\\
&=&\frac {\tau (w)}{2^i}h_i,
\end{eqnarray*}

if $h=\sum_{i=1}^\infty h_ie_i.$ We conclude that $(b(w))^\star_{\mathcal L(\ell^2)} (h)=\overline{\tau (w)}\sum_{i=1}^\infty \frac {h_i}{2^i}e_i.$ In particular, $(b(w))^\star_{\mathcal L(\ell^2)} (e_i)=\frac {\overline{\tau (w)}}{2^i}e_i$ for every positive integer $i.$ Comparing with the definition of $b$ at the beginning of Section \ref{sec6}:
$$b(w)(e_i)=\frac {\tau (w)}{2^i}e_i$$ 
for every positive integer $i,$ we easily obtain from (1) that $$(b(\cdot))^\star_{\mathcal L(\ell^2)}\in BMO^1_{\alpha} ({\mathbb B}_n, \mathcal L(\ell^2, \ell^1)).$$
\item
We finally show that $\sup \limits_{e\in \mathbb C^d, \left \Vert e\right \Vert_{\mathbb C^d}=1} \lim \limits_{z\rightarrow \partial {\mathbb B}_n} \left \Vert T_b\left (k_z^\alpha e \right )\right \Vert_{A^2_\alpha (\mathbb B_n,  \ell^2)}=0$ for every positive integer $d.$ For $e=\sum_{j=1}^d \lambda_j e_j,$ we have:
\begin{eqnarray*}
T_b\left (k_z^\alpha e\right )(\zeta)
&=&\int_{{\mathbb B}_n} b(w)(e)k_z^\alpha (w)\overline{K_\zeta^\alpha (w)}d\nu_{\alpha}(w)\\
&=&\int_{{\mathbb B}_n} \sum_{j=1}^d \frac {\lambda_j}{2^j} \tau (w)e_j k_z^\alpha (w)\overline{K_\zeta^\alpha (w}d\nu_{\alpha}(w)\\
&=&\sum_{j=1}^d \frac {\lambda_j}{2^j} \left (\int_{{\mathbb B}_n} \tau (w) k_z^\alpha (w)\overline{K_\zeta^\alpha (w)}d\nu_{\alpha}(w)\right )e_j\\
&=&\sum_{j=1}^d \frac {\lambda_j}{2^j} T^{\mathbb C}_\tau k_z^\alpha (\zeta) e_j.
\end{eqnarray*}

It follows that
$$
\left \Vert T_b\left (k_z^\alpha e \right )\right \Vert_{A^2_\alpha (\mathbb B_n,  \ell^2)}=\left (\sum_{j=1}^d \frac {\left \vert\lambda_j\right \vert^2}{2^{2j}} \right )^\frac 12 \left \Vert T^{\mathbb C}_\tau k_z^\alpha \right \Vert_{A^2_\alpha({\mathbb B}_n, \mathbb C)}.
$$
The scalar Toeplitz operator $T^{\mathbb C}_\tau$ is compact since $\widetilde \tau (z)\rightarrow 0$ as $z\rightarrow \partial {\mathbb B}_n.$ In this case, we obtain that $\lim \limits_{z\rightarrow \partial {\mathbb B}_n} \left \Vert T^{\mathbb C}_\tau k_z^\alpha\right \Vert_{A^2_\alpha({\mathbb B}_n, \mathbb C)}=0:$  this implies that 
$$\lim \limits_{z\rightarrow \partial {\mathbb B}_n} \left \Vert T_b\left (k_z^\alpha e \right )\right \Vert_{A^2_\alpha (\mathbb B_n,  \ell^2)}=0.$$
\end{enumerate}

\section{Characterization of general compact operators via the Berezin transform}\label{sec12}
In this  section, we state a corollary which should have been stated by Rahm at the end of his paper \cite{R2016}. 


Extending Definition \ref{def72} to the infinite dimension and in view of Definition \ref{def17} in the case $E=F=\ell^2$, we state the following definition.

 \begin{defn}\label{def122}
Let $S\in \mathcal L(A^2_\alpha (\mathbb B_n, \ell^2)).$ The Berezin transform $\widetilde S$ is the $\mathcal L(\ell^2)$-valued function defined by
$$\langle \widetilde S (z)e, h\rangle_{\ell^2}=\langle S\left (k_z^\alpha e\right ), k_z^\alpha h\rangle_{A^2_\alpha (\mathbb B_n, \ell^2)}.$$
\end{defn}

Combining \cite[Theorem 4.1, Lemma 4.5]{R2016}, we obtain the following theorem in our particular context of standard Bergman spaces on the unit ball $\mathbb B_n$ of $\mathbb C^n.$

\begin{thm}\label{thm123}
Let $S\in \mathcal T_{L^\infty_{fin}}.$ The following two assertions are equivalent.
\begin{enumerate}
\item
$$\lim \limits_{z\rightarrow \partial \mathbb B_n} \hskip 2truemm \left \vert \langle \widetilde S(z)e_j, e_i\rangle_{\ell^2}\right \vert=0$$
for all positive integers $j, i;$
\item
$$\lim \limits_{z\rightarrow \partial \mathbb B_n} \left \Vert S^z f\right \Vert_{A^2_\alpha \left (\mathbb B_n, \ell^2\right )}=0$$
for every $f\in A^2_\alpha \left (\mathbb B_n, \ell^2\right ).$
\end{enumerate}
\end{thm}

In \cite[Theorem 4.4]{R2016}, Rahm proved the following theorem.

\begin{thm}\label{thm124}
Let $S\in \mathcal L(A^2_\alpha \left (\mathbb B_n, \ell^2\right )).$ The following two assertions are equivalent.
\begin{enumerate}
\item
$S$ is compact;
\item
$\lim \limits_{z\rightarrow \partial \mathbb B_n} \hskip 2truemm \left \Vert\widetilde S(z)\right \Vert_{\mathcal L(\ell^2)}=0$ and $S\in \mathcal T_{L^\infty_{fin}}.$
\end{enumerate}
\end{thm}

To obtain the proof of the announced corollary, we shall use the following atomic decomposition theorem for functions in vector-valued Bergman spaces \cite[Theorem 2.3.5]{V2017}.

\begin{thm}\label{thm125}
	Let $\{a_j\}_{j=1}^\infty $ be a $r$-lattice in $\mathbb B_n$. If $f\in A^2_\alpha (\mathbb B_n, \ell^2),$ there is   a double sequence $\{\lambda_{i, j}\}_{i, j=1}^\infty $ of complex numbers such that
	$$f=\sum_{j=1}^\infty \sum_{i=1}^\infty \lambda_{i, j}k_{a_j}^\alpha e_i,$$
	where the convergence of the double series is with respect to the $A^2_\alpha (\mathbb B_n, \ell^2)$-norm and 
	\begin{equation}\label{est}
		\sum_{j=1}^\infty \sum_{i=1}^\infty  \left \vert \lambda_{i, j}\right \vert^2\lesssim \left \Vert f\right \Vert^2_{A^2_\alpha (\mathbb B_n, \ell^2)}.
	\end{equation}
\end{thm}

Combining Theorem \ref{thm123}, Theorem \ref{thm124} with Theorem \ref{thm118} \cite[Theorem 3.4, Corollary 3.5]{R2016}, we obtain the announced corollary.

\begin{cor}
Let $S\in \mathcal L(A^2_\alpha \left (\mathbb B_n, \ell^2\right )).$ The following four assertions are equivalent.
\begin{enumerate}
\item 
$\lim \limits_{z\rightarrow \partial \mathbb B_n} \hskip 2truemm \left \Vert\widetilde S(z)\right \Vert_{\mathcal L(\ell^2)}=0$ and $S\in \mathcal T_{L^\infty_{fin}}.$
\item  For any $f\in  A^2_\alpha (\mathbb B_n, \ell^2)$, we have
$ \lim \limits_{z\rightarrow \partial \mathbb B_n} \hskip 2truemm \left \Vert S^z f\right \Vert_{A^2_\alpha \left (\mathbb B_n, \ell^2\right )}=0$  and $S\in \mathcal T_{L^\infty_{fin}}.$
\item
$S$ is the limit in the operator norm of a sequence $\{S_m\}_{m=1}^\infty$ of sufficiently localized operators which satisfy the following two conditions.
\begin{enumerate}
\item[(i)]
$\lim \limits_{d\rightarrow \infty} \left \Vert S_m M_{I_{(d)}}\right \Vert_{\mathcal L(A^2_\alpha \left (\mathbb B_n, \ell^2\right ))}=0$ for every positive integer $m;$
\item[(ii)]
$\lim \limits_{m\rightarrow \infty}\sup \limits_{e\in \mathbb C^d, \left \Vert e\right \Vert_{\mathbb C^d}=1} \limsup \limits_{z\rightarrow \partial \mathbb B_n} \hskip 2truemm \left \Vert S_m (k_z^\alpha e)\right \Vert_{A^2_\alpha \left (\mathbb B_n, \ell^2\right )}=0,$ \quad  for every positive integer $d.$

\end{enumerate}
\item
$S$ is compact on $A^2_\alpha \left (\mathbb B_n, \ell^2\right ).$
\end{enumerate}
\end{cor}

\begin{proof}
$(1)\Rightarrow (2)$ Apply Theorem \ref{thm123}.\\
$(2)\Rightarrow (3)$
For a general operator $S\in \mathcal L(A^2_\alpha (\mathbb B_n,  \ell^2)),$ Rahm \cite{R2016} showed in the proof of his Corollary 3.5 (Theorem \ref{thm115} above) that the assumption $S\in \mathcal T_{L^\infty_{fin}}$ implies that $S$ is the limit in the operator norm of a sequence $\{S_m\}_{m=1}^\infty$ of sufficiently localized operators which satisfy the condition
$\lim \limits_{d\rightarrow \infty} \left \Vert S_m M_{I_{(d)}}\right \Vert_{\mathcal L(A^2_\alpha \left (\mathbb B_n, \ell^2\right ))}=0$
for all positive integers $m.$  Moreover, for all  positive integers $d$ and $m,$  every $e\in \mathbb C^d$ such that $\left \Vert e\right \Vert_{\mathbb C^d}=1,$ and every $z\in \mathbb B_n,$  
we have:
\begin{eqnarray*}
	\left \Vert S_m (k_z^\alpha e)\right \Vert_{A^2_\alpha \left (\mathbb B_n, \ell^2\right )}
	&\leq & \left \Vert S(k_z^\alpha e) \right \Vert_{A^2_\alpha \left (\mathbb B_n, \ell^2\right )}+\left \Vert \left (S_m-S\right )(k_z^\alpha e)\right \Vert_{A^2_\alpha \left (\mathbb B_n, \ell^2\right )}\\
	&\leq & \left \Vert S(k_z^\alpha e)\right \Vert_{A^2_\alpha \left (\mathbb B_n, \ell^2\right )}+\left \Vert S_m-S \right \Vert_{\mathcal L(A^2_\alpha \left (\mathbb B_n, \ell^2\right ))}.
\end{eqnarray*}
Since $\lim \limits_{m\rightarrow \infty} \left \Vert S_m-S \right \Vert_{\mathcal L(A^2_\alpha \left (\mathbb B_n, \ell^2\right ))}=0,$ then for every $\epsilon \textgreater 0,$ there exists a positive integer $m_0=m_0 (\epsilon)$ such that for every positive integer $m\textgreater m_0,$ the following implication holds.
\begin{equation}\label{equ121}
	m\textgreater m_0 \Rightarrow \left \Vert S_m-S \right \Vert_{\mathcal L(A^2_\alpha \left (\mathbb B_n, \ell^2\right ))} \textless \epsilon.
\end{equation}
Then for every $m\textgreater m_0,$ for every positive integer $d$  and  every  $e\in \mathbb C^d$ such that $\left \Vert e\right \Vert_{\mathbb C^d}=1,$  
 we obtain:
$$\limsup \limits_{z\rightarrow \partial \mathbb B_n} \hskip 2truemm \left \Vert S_m (k_z^\alpha e)\right \Vert_{A^2_\alpha \left (\mathbb B_n, \ell^2\right )}\leq \limsup \limits_{z\rightarrow \partial \mathbb B_n} \left \Vert S(k_z^\alpha e)\right \Vert_{A^2_\alpha \left (\mathbb B_n, \ell^2\right )}+\epsilon=0+\epsilon=\epsilon,$$
because in view of $(2),$ we have 
$\lim\limits_{z\rightarrow \partial \mathbb B_n} \left \Vert S(k_z^\alpha e)\right \Vert_{A^2_\alpha \left (\mathbb B_n, \ell^2\right )}=0.$ 
We conclude that
$$\sup \limits_{e\in \mathbb C^d, \left \Vert e\right \Vert_{\mathbb C^d}=1} \limsup \limits_{z\rightarrow \partial \mathbb B_n} \hskip 2truemm \left \Vert S_m (k_z^\alpha e)\right \Vert_{A^2_\alpha \left (\mathbb B_n, \ell^2\right )}\leq \epsilon$$
for every $m\textgreater m_0$ and for every positive integer $d.$  
In other words, 
$$\lim \limits_{m\rightarrow \infty}\sup \limits_{e\in \mathbb C^d, \left \Vert e\right \Vert_{\mathbb C^d}=1} \limsup \limits_{z\rightarrow \partial \mathbb B_n} \hskip 2truemm \left \Vert S_m (k_z^\alpha e)\right \Vert_{A^2_\alpha \left (\mathbb B_n, \ell^2\right )}=0$$ for every positive integer $d.$

$(3)\Rightarrow (4)$ 
According to Theorem \ref{thm125}, for every $\epsilon > 0,$ there are a finite double subsequence $\{\lambda_{i,j},\hskip 1truemm i=1,\cdots,I;\hskip 1truemm j=1,\cdots, J\},$ a finite subsequence  $\{a_j: j=1,\cdots, J \}$ of points of $\mathbb B_n$ and a function $g\in A^2_\alpha (\mathbb B_n,  \ell^2)$ such that
$$f=\sum_{j=1}^J \sum_{i=1}^I \lambda_{i, j} k_{a_j}^\alpha e_i+g,$$
with 
$\left \Vert g\right \Vert_{A^2_\alpha (\mathbb B_n,  \ell^2)}< \epsilon.$ Then for every $z\in \mathbb B_n,$ 
we have
$$S_m U_z f=\sum_{j=1}^J \sum_{i=1}^I\lambda_{i,j} S_m U_z (k_{a_j}^\alpha e_i)+S_m U_z g$$
and consequently
\begin{eqnarray*}
	\left \Vert S_m U_z f\right \Vert_{A^2_\alpha (\mathbb B_n,  \ell^2)}
	&\leq & \sum_{j=1}^J \sum_{i=1}^I \left \vert \lambda_{i,j} \right \vert \left \Vert S_m U_z (k_{a_j}^\alpha e_i)\right \Vert_{A^2_\alpha (\mathbb B_n,  \ell^2)}+\left \Vert S_m\right \Vert_{\mathcal L(A^2_\alpha (\mathbb B_n,  \ell^2))} \left \Vert g\right \Vert_{A^2_\alpha (\mathbb B_n,  \ell^2)}\\
	&<&  \sum_{j=1}^J \sum_{i=1}^I \left \vert \lambda_{i,j} \right \vert \left \Vert S_m U_z(k_{a_j}^\alpha e_i)\right \Vert_{A^2_\alpha (\mathbb B_n,  \ell^2)}+\left \Vert S_m\right \Vert_{\mathcal L(A^2_\alpha (\mathbb B_n,  \ell^2))}\epsilon.
\end{eqnarray*}
.

We record the following identity. For $z, a\in \mathbb B_n,$ there exists a complex unimodular number $\gamma_{z, a}$ such that
$$U_z k_a^\alpha = \gamma_{z, a} k^\alpha_{\varphi_z (a)}.$$
More explicitly, $\gamma_{z, a}=\frac{\left \vert 1-\langle z, a\rangle\right \vert}{1-\langle z, a\rangle}.$
In view of (\ref{equ121}), 
we take $m$ so large so that $\left \Vert S_m\right \Vert_{\mathcal L(A^2_\alpha (\mathbb B_n,  \ell^2))}\leq 2\left \Vert S\right \Vert_{\mathcal L(A^2_\alpha (\mathbb B_n,  \ell^2))}.$ 
Hence
$$\left \Vert S_m U_z f\right \Vert_{A^2_\alpha (\mathbb B_n,  \ell^2)}\leq 2\left \Vert S\right \Vert_{\mathcal L(A^2_\alpha (\mathbb B_n,  \ell^2))}\epsilon + \sum_{j=1}^J \sum_{i=1}^I \left \vert \lambda_{i,j} \right \vert \left \Vert S_m( k_{\varphi_z (a_j)}^\alpha e_i)\right \Vert_{A^2_\alpha (\mathbb B_n,  \ell^2)}.$$
We point out that for every $a\in \mathbb B_n, \hskip 2truemm \varphi_z (a) \rightarrow \partial \mathbb B_n$ as $z \rightarrow \partial \mathbb B_n.$ In view of assertion (ii) of (3), for all $i\in \{1,\cdots,I\}$ and  $j\in \{1,\cdots, J\},$ there exists a positive integer $m_0$ such that
$$\limsup \limits_{z \rightarrow \partial \mathbb B_n} \left \Vert S_m (k_{\varphi_z (a_j)}^\alpha e_i)\right \Vert_{A^2_\alpha (\mathbb B_n,  \ell^2)}\textless \frac \epsilon{\sqrt{IJ}}$$
for all $m\textgreater m_0.$
We then obtain 
$$\limsup \limits_{z \rightarrow \partial \mathbb B_n}\left \Vert S_m U_z f\right \Vert_{A^2_\alpha (\mathbb B_n,  \ell^2)}\leq 2\left \Vert S\right \Vert_{\mathcal L(A^2_\alpha (\mathbb B_n,  \ell^2))}\epsilon +\frac \epsilon{\sqrt{IJ}}\sum_{j=1}^J \sum_{i=1}^I  \left \vert \lambda_{i, j}\right \vert.
$$
By the Schwarz inequality, we obtain:
$$\sum_{j=1}^J \sum_{i=1}^I  \left \vert \lambda_{i, j}\right \vert\leq \left (\sum_{j=1}^\infty \sum_{i=1}^\infty  \left \vert \lambda_{i, j}\right \vert^2 \right )^{\frac 12}\sqrt{IJ}.$$
So for every $f\in A^2_\alpha (\mathbb B_n,  \ell^2),$ we get
$$\limsup \limits_{z \rightarrow \partial \mathbb B_n}\left \Vert S_m U_z f\right \Vert_{A^2_\alpha (\mathbb B_n,  \ell^2))}\leq 2\left \Vert S\right \Vert_{\mathcal L(A^2_\alpha (\mathbb B_n,  \ell^2))}\epsilon +\left (\sum_{j=1}^\infty \sum_{i=1}^\infty  \left \vert \lambda_{i, j}\right \vert^2 \right )^{\frac 12}\epsilon.$$
We next assume that $\left \Vert f\right \Vert^2_{A^2_\alpha (\mathbb B_n, \ell^2)}=1.$ It follows from (\ref{est}) that
$$\left (\sum_{j=1}^\infty \sum_{i=1}^\infty  \left \vert \lambda_{i, j}\right \vert^2\right )^{\frac 12}\lesssim 1.$$
In view of assertion i) of Theorem \ref{thm118}, this implies that
$$\left \Vert S_m\right \Vert_e\simeq \sup \limits_{f\in A^2_\alpha (\mathbb B_n,  \ell^2): \left \Vert f\right \Vert_{A^2_\alpha (\mathbb B_n,  \ell^2)}=1} \limsup \limits_{z\rightarrow \partial \mathbb B_n} \left \Vert S_m^z f\right \Vert_{A^2_\alpha (\mathbb B_n,  \ell^2)}\lesssim \left (\left \Vert S\right \Vert_{\mathcal L(A^2_\alpha (\mathbb B_n,  \ell^2))}+1\right )\epsilon$$
for $m$ large enough ($m> m_0).$ For such a $m,$ we have
\begin{eqnarray*}
	\left \Vert S\right \Vert_e
	&\leq& \left \Vert S_m\right \Vert_e+\left \Vert S_m-S\right \Vert_{\mathcal L(A^2_\alpha (\mathbb B_n,  \ell^2))}\\
	&\lesssim &\left (\left \Vert S\right \Vert_{\mathcal L(A^2_\alpha (\mathbb B_n,  \ell^2))}+1\right )\epsilon+\epsilon\\
	&\lesssim & \left (\left \Vert S\right \Vert_{\mathcal L(A^2_\alpha (\mathbb B_n,  \ell^2))}+2\right )\epsilon
\end{eqnarray*}
by (\ref{equ121}) and the previous estimate.
Since $\epsilon \textgreater 0$ is arbitrary, we conclude that $\left \Vert S\right \Vert_e=0,$ i.e. $S$ is compact on $A^2_\alpha (\mathbb B_n, \ell^2).$\\

$(4)\Rightarrow (1)$
Apply the implication $(1)\Rightarrow (2)$ of Theorem \ref{thm124}.
\end{proof}

\section{An open question}\label{sec13}
For $d=1,$ Xia \cite{X2015} proved the reverse implication: if $S$ is sufficiently localized (and even weakly localized), then $S\in \mathcal T^1_\alpha.$ It would be interesting to extend this result  to the infinite dimension.  More specifically, for Toeplitz operators, if $S=T_b$ such that $b, \left (b(\cdot)\right )^\star_{\mathcal L(\ell^2)} \in BMO^1_{\alpha} ({\mathbb B}_n, \mathcal L(\ell^2, \ell^1\cap \ell^2))$  and $\widetilde b \in L^\infty ({\mathbb B}_n, \mathcal L(\ell^2)),$ does one have $T_b\in \mathcal T_{L^\infty_{fin}}$ whenever $\lim \limits_{d\rightarrow \infty} \left \Vert T_bM_{I(d)}\right \Vert_{\mathcal L(A^2_\alpha (\mathbb B_n,  \ell^2))}=0?$ This would also be an infinite dimensional generalization of the result of Sadeghi-Zorboska (Theorem \ref{thm110} above).

\section*{Acknowledgement} The authors wish to express their gratitude to Robert Rahm for his assistance in arranging the proof of Proposition \ref{pro24}.

\end{document}